\documentclass[a4paper,11pt,final]{article}
\usepackage{amsmath, amsthm, amssymb}
\usepackage{latexsym}
\usepackage[latin1]{inputenc}
\usepackage{epsf,exscale,times}
\usepackage[T1]{fontenc}
\usepackage[dvips]{graphicx}
\usepackage{graphics,epsfig}
\usepackage{epsfig,rotate,color}
\usepackage{showkeys}
\usepackage{subfigure}  

\newcommand{\SR}{\mathcal{S}(\mathbb{R})}

\newcommand{\I}{\mathcal{I}}

\newcommand{\F}{\mathcal{F}}
\newcommand{\R}{\mathbb{R}}

\newtheorem{proposition}{Proposition}

\newtheorem{remark}{Remark}

\title{Simultaneous denoising and enhancement of  signals by a fractal conservation law}

\author{Pascal Azerad, Afaf Bouharguane   and Jean-Fran\c cois Crouzet \footnote{Institut de Math\'ematiques et Mod\'elisation de Montpellier, UMR 5149 CNRS, Universit\'e Montpellier 2, Place Eug\`ene Bataillon, CC 051
34095 Montpellier, France. Email:\, {\sffamily azerad@math.univ-montp2.fr},\, {\sffamily bouharg@math.univ-montp2.fr},\,
{\sffamily crouzet@math.univ-montp2.fr}} 
}
\date{\today}

\begin{document}
\maketitle

%\begin{quote} \footnotesize
%\noindent \textsc{Abstract.} 
\begin{abstract}
In this paper, a new filtering method is presented for simultaneous noise reduction and enhancement of signals using a fractal scalar conservation law which is simply the forward heat equation modified by a fractional anti-diffusive term of lower order.
This kind of equation has been first introduced by physicists to describe morphodynamics of sand dunes.
To evaluate the performance of this new filter, we perform a number of numerical tests on various signals. Numerical simulations are based on finite difference schemes or Fast Fourier Transform.
We used two well-known measuring metrics in signal processing for the comparison.
The results indicate that the proposed method outperforms the well-known Savitzky-Golay filter in signal denoising.
Interesting  multi-scale properties w.r.t. signal frequencies are exhibited allowing  to control both denoising and contrast enhancement. 
\end{abstract}

%\end{quote}

%\vspace{5mm}


\section{Introduction}

Filtering is a process that removes some unwanted component from a signal. It is a very important task in signal processing, data analysis and communication systems. Many techniques have been proposed for this purpose. For instance, we can use simple  averaging filters such  as a Gaussian filter. It is well-known that this filter can be realized by solving the heat equation. Other more general partial differential equations (PDE) have been used, with  non-linear anisotropic diffusion \cite{catte,perona},  non-linear fractional diffusion \cite{bai,patrick} or fractional time-derivative \cite{kirane}. It has been proved that PDEs are suitable in signal denoising \cite{nader}. The denoising applications have to take into account two points. First, we want to obtain a clean and readily observable signal (improving signal-to-noise ratio SNR) and secondly, preserve  the original shape characteristics (maxima, minima...) of the signal. This task is complex because it is very important that the denoising has no blurring effect on the images and does not change the location of image edges,  \cite{buades}.
For some applications, it is interesting to amplify some features of the signal such as relative maxima or minima in order to enhance its contrast. Usually, these features are flattened by the denoising methods based on averaging techniques.  
%\textcolor{red}{Currently, there are a number of methods for signal denoising but the techniques for enhancement of signals seems more scarce}.
Among the  denoising methods  which preserve characteristics of the initial signal is the Savitzky-Golay filter, with which we will compare our new method. 

A basic and crude idea to enhance the contrast could be to solve the backward heat equation for a few time steps. Of course, this is an ill-posed PDE and we do not advocate this unsafe method but it illustrates the fact that enhancement and denoising are antagonistic operations. 
Our method is based on a linear PDE, with two antagonistic terms : a usual diffusion and a nonlocal fractional anti-diffusive term of lower order.
The diffusion is used to reduce the noise whereas the nonlocal anti-diffusion is used to enhance the contrast.  Let us emphasize that we perform 
\emph{at the same time} noise reduction and contrast enhancement.
Our  method is based on  the Cauchy problem of the following PDE :
\begin{equation}
\begin{cases}
\partial_t u(t,x) - a \, \partial_{xx}^2 u(t,x) + b \, \I_\lambda [u(t,.)] (x)
 = 0 & t \in (0,T), x \in \mathbb{R}, \\
u(0,x)=u_0(x) & x \in \mathbb{R},
\end{cases}
\label{EDP}
\end{equation}
where $T$ is any given positive time, $u_0 \in L^2(\R)$, $a, \, b$ are positive constants and $\I_\lambda$
is a fractional operator defined as follows via the Fourier transform:
for any Schwartz function $\varphi \in \SR$
\begin{equation}
\I_\lambda [\varphi](x) := - \, \F^{-1}(|.|^{\lambda} \F(\varphi))(x)\,
\end{equation}
where $0 < \lambda < 2$ and 
$\F$ denote the Fourier transform defined by: for all $\xi \in \R$
\begin{equation*}
\F f (\xi):= \int_{\R} e^{-2i\pi x \xi} f(x) dx.
\end{equation*}

Let us note that for $\lambda \in ]1,2[$, we have an explicit nonlocal formula (see proposition \ref{khinchine})

\begin{equation}
\I_\lambda [\varphi](x) = \alpha_\lambda \int_\R \frac{\varphi^{''}(x-\xi)}{|\xi|^{\lambda-1}} d\xi
\end{equation}
where $\alpha_\lambda$ is a suitable constant. \\
Alternatively, we can also give a slightly different definition, inspired by fractional calculus (see remark below) :
\begin{equation}
\I_\lambda [\varphi](x) = \alpha_\lambda \int_0^{+\infty} \frac{\varphi^{''}(x-\xi)}{|\xi|^{\lambda-1}} d\xi.
\label{liouville}
\end{equation}

\begin{remark}
For causal signals (i.e $\varphi(x)=0$ for $x<0$) we have

\begin{equation}
\frac{1}{\Gamma(2-\lambda)} \int_0^{+\infty} \frac{\varphi^{''}(x-\xi)}{|\xi|^{\lambda-1}} d\xi= \frac{d^{\lambda - 2}}{d x^{\lambda -2}} \varphi'' = 
\frac{d^{\lambda}}{d x^{\lambda}} \varphi
\end{equation}
 which is exactly the Riemann Liouville definition of the fractional derivative \cite{pod}.
\end{remark} 
 
%with $D_\lambda$ a suitable constant, see  (\ref{expression2}) and $0 < \lambda < 2$. 
\begin{remark}
Our model is closely related to a nonlocal conservation law first introduced  to describe the morphodynamics of sand dunes  and ripples  sheared by  a fluid flow.  
Namely, Fowler (\cite{fowler2,fowler3,fowler1}) introduced the following equation \label{fowler1}
\begin{equation}
\partial_t u(t,x) + \partial_x\left(\frac{u^2}{2}\right) (t,x) + \I [u(t,.)] (x)
- \partial_{xx}^2 u(t,x) = 0, 
\label{fowlereqn}
\end{equation}
where $u=u(t,x)$ represents the dune height and $\I$
is a nonlocal operator defined as follows: for any Schwartz
function $\varphi \in \SR$ and any $x \in \mathbb{R}$,
\begin{equation}
\I [\varphi] (x) := \int_{0}^{+\infty} |\zeta|^{-\frac{1}{3}}
\varphi''(x-\zeta) d\zeta . \label{nonlocalterm}
\end{equation}
Equation \eqref{fowlereqn} is valid for a river flow over a erodible bottom $u(t,x)$ with slow variation. The nonlocal term appears after a subtle modeling of the basal shear stress. See \cite{alibaud,alvarez} for theoretical results on this equation.\\
This nonlocal term appears also in the work of Kouakou \& Lagree \cite{kouakou,lagree}.
The operator $\I[u]$ is a weighted mean of second
derivatives of $u$ with the bad sign and has therefore an anti-diffusive effect and create instabilities which
are controlled by the diffusive operator $-\partial_{xx}^2$. We find again this phenomenon for the equation \eqref{EDP}.
 
\end{remark}

The solution of the linear PDE \eqref{EDP} is obtained by convolution with the kernel $K$ of $\I_\lambda-\partial_{xx}^2$. Thereafter, this kernel will be our filter for denoising and enhancement of signals. The analysis of this kernel shows that the low frequencies are preserved, the medium frequencies are amplified and the high frequencies are eliminated. 
%Therefore, we investigate a new low pass filter with the particularity that this filter amplify, and not only retain, the low frequencies. 
It is clear that this kernel depends on the parameters $a, b, \lambda$ and that the choice of these coefficients will determine the quality of the noise reduction and of the enhancement. In this paper, we discuss the choice of these parameters. \\ 
To evaluate the performance of our filter,  we discretize the PDE  by two methods : the fast Fourier transform and the finite difference method. 
%The implementation is done with Matlab.  
Numerical studies of  nonlocal equations are rather scarce: among them, we mention \cite{JD2} which proves the convergence of a finite volume method to approximate the solutions of a fractal scalar conservation law, that is to say a conservation law regularized  by a  \emph{diffusive} fractional power of the Laplace operator and \cite{afaf}  which analyzes the stability  of  finite difference schemes for the solution of \eqref{fowlereqn}. 
In these works, the discretization of the fractal operator is performed  via an integral formula for $\I_\lambda$ similar to the ones appearing in \cite{alibaud, JD1}.
Our method is then compared with the Savitzky-Golay filter. Results show that the PDE \eqref{EDP} is relevant and effective for denoising with preservation or enhancement of features of the signal. \\

The remaining of this paper is organized as follows: in the next section, we explicit the solution of problem \eqref{EDP} and we give some properties of the kernel.    
In section \ref{DF}, we present explicit numerical schemes which approximate the fractal conservation law \eqref{EDP} and we give some numerical simulations. In this section, we also briefly present the Savitzky-Golay filter and compare it with our filter.
Other numerical simulations based on the FFT are given in section \ref{fft}. We also discuss the choice of parameters $a,b, \lambda$ and we highlight both the ability of denoising and contrast enhancement of our model. 
Section \ref{metrique} is devoted to the performance evaluation and metrics. Section \ref{conclusion} gives conclusions about this study.

\section{Theoretical study of the PDEs}

In this section, we verify that \eqref{EDP} is well-posed and in subsection \ref{sectionkernel} we analyze the properties of the kernel $K$ of $\I_\lambda-\partial_{xx}^{2}$, for $0< \lambda <2$.

\subsection{Well-posedness of the problem} 

Using the Fourier transform, we see that any solution to \eqref{EDP} satisfies the formula \eqref{duhamel}. 

\begin{proposition} 
Let $T>0$ and $u_{0} \in L^{2}(\R)$. The function $u \in L^\infty((0,T);L^2(\R)) $ is a solution of \eqref{EDP} if for any $t \in \left(0,T\right)$: 
\begin{equation}
u(t,x)=K(t,\cdot)\ast u_{0}(x)
\label{duhamel}
\end{equation}
where $K(t,x)=\F^{-1}\left(e^{-t \psi(\cdot)}\right)(x)$  with $\psi (\xi):= 4 \pi^2 a \xi^2 -b \, |\xi|^{\lambda} $
is the kernel of the operator $\I_\lambda-\partial^2_{xx}$. 
\end{proposition}

\begin{proposition}
Let $T>0$, $u_{0} \in L^{2}(\R)$. Then, the function 
\begin{equation*}
u: t \in (0,T] \rightarrow K(t,\cdot) \ast u_0 
\end{equation*}
is well-defined and belongs to $C([0,T];L^2(\R))$: $u$ is extended at $t=0$ by the value $u(0,.)=u_0$. 
\end{proposition}

\begin{proof}

We have
$$K(t,x) = \int_{\R} e^{2i\pi x \xi} \,h(\xi)\, d\xi,
$$
where $h(\xi) = e^{-t \psi(\xi)} $. For $0< \lambda <2$, it is easy to verify that  $h$  belongs to $W^{2,1}(\R)$
where $W^{2,1}(\R):=\left\{v \in L^{1}(\R); \frac{\partial v}{\partial x}, \frac{\partial^2 v}{\partial x^2} \in L^{1}(\R) \right\}$ therefore we have that  $K(t,.) \in L^{1}(\R)$. Hence, $\forall t>0$, $K(t,\cdot)*u_0$ is in $L^{2}(\R)$. 

Let us prove the strong continuity ie
\begin{equation*}
\lim_{t \to 0} K\left(t,.\right)*u_{0}=u_{0} \hspace{0.2 cm} in \hspace{0.2 cm} L^2\left(\R\right).
\end{equation*}
By Plancherel's formula,
\begin{multline} \label{ref strong continuity}
||K(t,\cdot) \ast u_0 -u_0||^2_{L^2(\R)}=||\F (K(t,\cdot) \ast u_0) -\F u_0||^2_{L^2(\R)}\\
=||e^{-t \psi} \F u_0 -\F u_0||^2_{L^2(\R)}=
\int_\R |e^{-t \psi}-1|^2\, |\F u_0|^2.
\end{multline}
Since the function $|e^{-t \psi}-1|^2\, |\F u_0|^2 $ converges
pointwise to $0$ on $\R$, as $t \rightarrow 0$ and as $\min \psi$ is finite then, by the dominated
convergence theorem, the last term of \eqref{ref
strong continuity} tends to $0$ as $t \rightarrow 0$. 

\end{proof}
\begin{remark}[Regularity of the solution] 
It is easy to see that $u \in C^{\infty}( (0,T] \times \R)$ because $K $ is smooth. The smoothness of $K$ is an immediate consequence of the theorem of derivation under the integral sign applied to the definition of $K$ by Fourier transform. To obtain the regularity at $t=0$, we have to suppose that the initial condition $u_0$ is in $C^{\infty}(\R)$ and satisfies for all $k \in \mathbb{N}$, $u_0^{(k)} \in L^{2}(\R)$. The proof is similar with replacing $u_0$ by $u_0^{(k)}$.   
\end{remark}

\begin{remark}
From formula \eqref{duhamel}, it is straightforward that
\begin{equation}
\int_{\R} u(t, x) \, dx = \F \left(K(t,\cdot)\right) (0) \cdot \int_{\R} u_0 (x) \, dx = \int_{\R} u_0 (x) \, dx.
\label{conserve}
\end{equation}
This is a ``mass conservation'' property. 
\end{remark}

\subsection{Study of the kernel \label{sectionkernel}}

In this subsection, we give some properties on the kernel $K$ of $ \I_{\lambda}-\partial_{xx}^2$.

\begin{proposition}
The kernel $K$ has a non-zero negative part.
\end{proposition}

\begin{proof}
Let us assume that $K$ is nonnegative, then
\begin{eqnarray*}
| e^{-t \psi(\xi)} | &\leq& || \F^{-1}(e^{-t \psi})||_{L^{1}(\R)}=\int_\R |K(t,.)| \\
&=& \int_\R K(t,.)= \F(K(t,.))(0)=e^{-t \psi (0)}=1   
\end{eqnarray*}
for all $\xi \in \R$.
But we have for $0< |\xi| < \left(\frac{b}{4 \pi^2 a}\right)^{\frac{1}{2-\lambda}}$, $|e^{-t \psi(\xi)}|=e^{-t(4 \pi^2 a \xi^2-b|\xi|^\lambda)}>1$ ,
this yields a contradiction.
%2) The smoothness of $K$ is an immediate consequence of the theorem of derivation under the integral sign applied to the definition of $K$ by Fourier transform.
\end{proof}
The main consequence of the non-positivity of $K$ is the failure of the maximum principle for the equation \eqref{EDP} \cite{alibaud}. Thereby,
$u (x, t)$ is not forced to remain in the interval $[\min (u_0), \max (u_0)]$. 
The signal can then be amplified when it is necessary. \\
Using the proof of the previous proposition, we can deduce that the enhancement of the frequencies will be feasible only for the low/medium frequencies. Of course, this amplification will depend on the choice of the parameters $a, b, \lambda$ and of the time $t$.
\begin{figure}[h!]
	\centering
        \subfigure[ The kernel of $\I_\lambda$ for $t=0.1$ (red) and $t=0.5$ s (blue) ]
	{\includegraphics[scale=0.5]{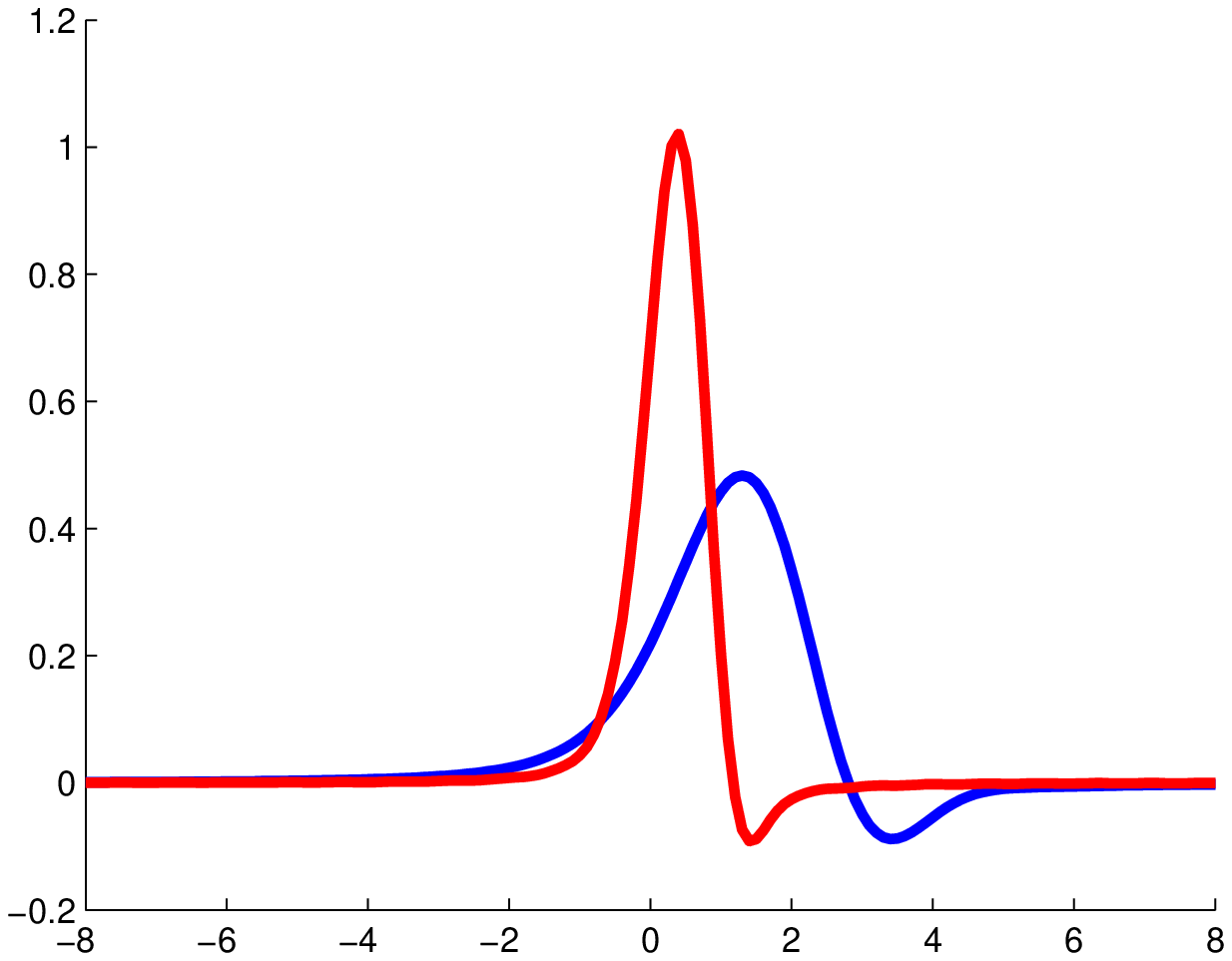} }
	\subfigure[ Behaviour of $\psi$  ]
	{\includegraphics[scale=0.5]{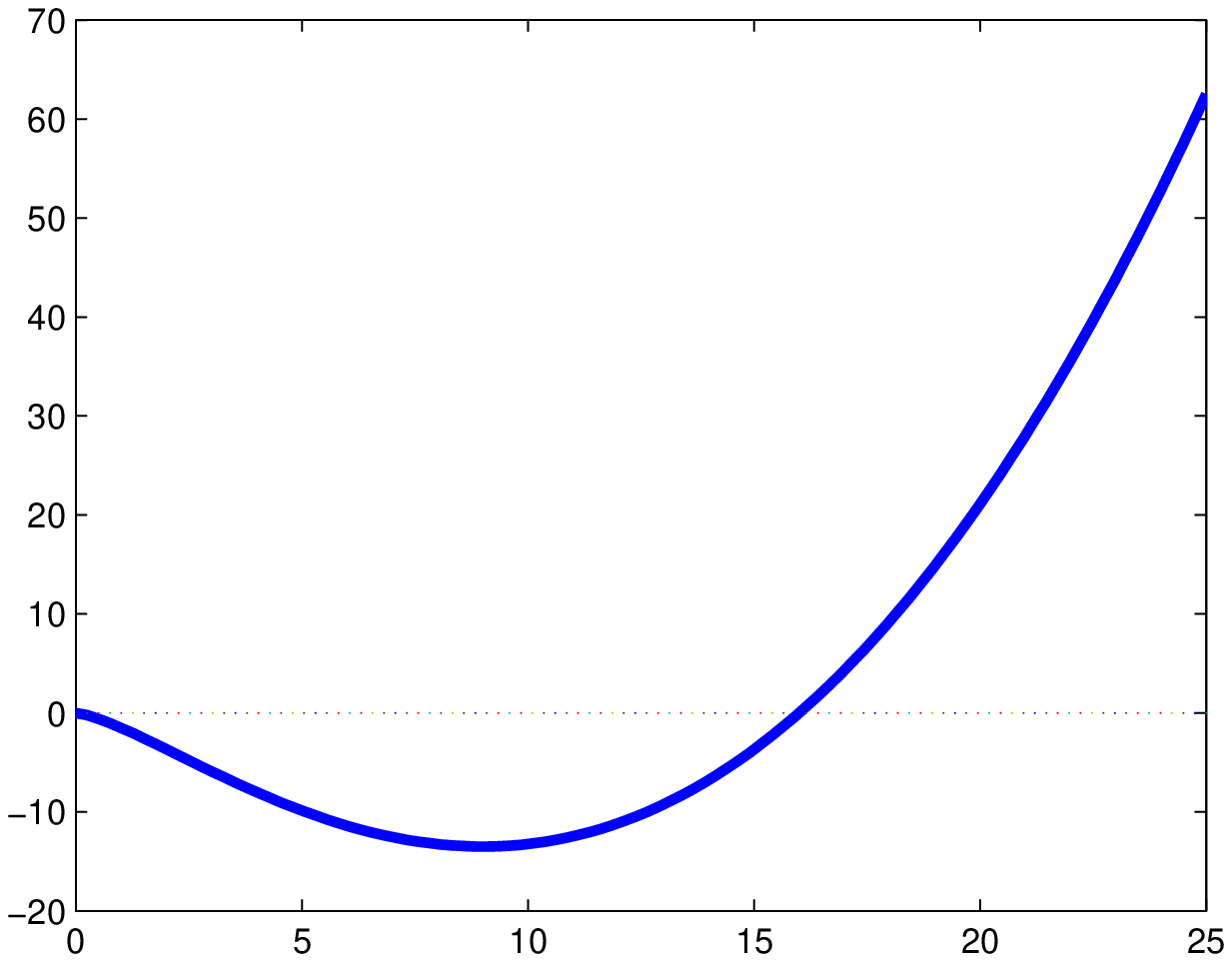} }
        \caption{ $a=0.5, b=2$ and $\lambda=1.5$ \label{kernel}
}
\end{figure} 

We expose in figure \ref{kernel} the evolution of $K(t,.)$ for different times and the behaviour of $\psi$ for $a=0.5, b=2, \lambda=1.5$.

\section{Finite difference schemes  \label{DF}} 

In this section, we present a finite difference numerical scheme to directly approximate the solution to \eqref{EDP} for any $\lambda \in ]0,2[$.  \\ 

\subsection{Integral representation of $\I_\lambda$ \label{sectionintegral}}

To approximate the nonlocal term, it is useful to give an integral formula for $\I_\lambda$: 

\begin{proposition}[\cite{JD3, imbert}] 
If $\lambda \in ]0,2[$, we have for all $\varphi \in \SR$, all $x \in \R$ and all $r>0$,
\begin{equation*}
\I_\lambda[\varphi](x)=C_\lambda \left(\int_{-r}^{r} \frac{\varphi(x+z)-\varphi(x)-\varphi^{'}(x)z}{|z|^{1+\lambda}} \, dz + \int_{\R \setminus (-r,r)} \frac{\varphi(x+z)-\varphi(x)}{|z|^{1+\lambda}} \, dz  \right)
\end{equation*}
where $C_\lambda=\frac{\lambda \Gamma(\frac{1+\lambda}{2})}{2 \pi^{ {\frac{1}{2}+\lambda}} \Gamma(1-\frac{\lambda}{2})}$ and $\Gamma $ denote the Euler function.
%\textcolor{red}{This formula can be generalized as follows:}
\begin{itemize}
\item[1)] If $\lambda \in ]0,1[ $ then
\begin{equation}
\I_\lambda[\varphi](x)=C_\lambda \int_\R \frac{\varphi(x+z)-\varphi(x)}{|z|^{1+\lambda}} \, dz.
\label{expression0}
\end{equation}

\item[2)] If $\lambda \in ]1,2[ $ then
\begin{equation}
\I_\lambda[\varphi](x)=C_\lambda \int_\R \frac{\varphi(x+z)-\varphi(x)-\varphi^{'}(x)z}{|z|^{1+\lambda}} \, dz.
\label{expression1}
\end{equation}

\end{itemize}
\end{proposition}

From this proposition, we deduce the following useful result :

\begin{proposition}
\label{khinchine}
For all $\varphi \in \SR$ and all $x \in \R$ we have for $\lambda \in ]1,2[ $ 
\begin{equation}
\I_{\lambda}[\varphi](x)= \alpha_\lambda \int_\R \frac{\varphi^{''}(x+z)}{ |z|^{\lambda-1}} \, dz= \alpha_\lambda \int_\R \frac{\varphi^{''}(x-z)}{ |z|^{\lambda-1}} \, dz
\label{expression2}
\end{equation}
where $\alpha_\lambda=\frac{C_\lambda}{\lambda(\lambda-1)}$.
\end{proposition}

\begin{proof}
The regularity of $\varphi$ ensures the validity of the following computations .
Since
\begin{eqnarray*}
 \varphi(x+z)-\varphi(x)-\varphi^{'}(x)z &=& \int_{0}^{z} \left(\varphi^{'}(x+y)-\varphi^{'}(x)\right) \, dy \\
 &=& \int_{0}^{1}\left( \varphi^{'}(x+tz)-\varphi^{'}(x)\right) \, z \, dt, 
\end{eqnarray*}
the last equality arises from the change of variable $y = t z$. Then, using Fubini's Theorem, we have
\begin{eqnarray*}
\I_{\lambda}(x) &=& C_\lambda \int_\R \int_{0}^{1} (\varphi^{'}(x+tz)-\varphi^{'}(x)) \, \frac{z}{|z|^{1+\lambda}} \, dt \, dz \\
&=& \int_{0}^{1} \left(\int_\R (\varphi^{'}(x+tz)-\varphi^{'}(x)) \, \frac{z}{|z|^{1+\lambda}} \, dz \right) dt \\
&\underbrace{=}_{y=tz}& C_\lambda  \int_{0}^{1} \left( \int_\R (\varphi^{'}(x+y)-\varphi^{'}(x)) \frac{y}{|y|^{1+\lambda}} \, dy \right) t^{\lambda-1} \, dt \\
&=& C_\lambda \int_{0}^{1} t^{\lambda-1} \, dt \int_\R (\varphi^{'}(x+y)-\varphi^{'}(x)) \frac{y}{|y|^{1+\lambda}} \, dy \\
&=& \frac{C_\lambda}{\lambda} \int_\R \left( \int_{0}^{1} \varphi^{''}(x+ sy) y \, ds \right) \frac{y}{ |y|^{1+\lambda}} \, dy \\
%&=& \frac{C_\lambda}{\lambda} \int_\R \int_{0}^{1} \varphi^{''}(x+ sy) y \, ds \, dy \\
&\underbrace{=}_{z=sy}& \frac{C_\lambda}{\lambda} \int_\R \frac{\varphi^{''}(x+ z)}{|z|^{\lambda-1}} \, dz \int_{0}^{1} s^{\lambda-2} \, ds \\
&=& \frac{C_\lambda}{\lambda (\lambda-1)} \int_\R \frac{\varphi^{''}(x+ z)}{|z|^{\lambda-1}} \, dz.
\end{eqnarray*} 
\end{proof}

%\begin{remark} For $\lambda= 4/3$, we find a linearized Fowler equation up to a constant of proportionality. 
%However, in the Fowler model, the nonlocal operator is an integral defined on $\R^{+}$.  
%\end{remark}

\subsection{The numerical scheme}
%In this section, we take $\lambda \in ]1,2[$.
The spatial discretization is given by a set of points ${x_{i}; i=1,...,N}$ and the discretization in time is represented by a sequence of times $t^0=0<...<t^n<...<T$. For the sake of simplicity we will assume constant step size $\delta x$ and $\delta t$ in space and time, respectively. The discrete solution at a point will be represented by $u^n_i \approx u(t^n,x_i)$. 
In this section, we will present the behaviour of explicit numerical scheme which directly approximate the solution to \eqref{EDP}.

We discretize all terms of the equation using an explicit method. We consider for any $n,j \in \mathbb{N}$
\begin{equation}\label{explicit}
\frac{u_{j}^{n+1}-u_{j}^{n}}{\delta t}- a \, \frac{u_{j+1}^{n}-2u_{j}^{n}+u_{j-1}^{n}}{\delta x^{2}}+ b \, \I_{\delta x}[u^{n}]_{j}=0
\end{equation} 
where $\I_{\delta x}$ is a discretization of the nonlocal term $\I_\lambda$. 
This scheme is explicit because the values of the solution at time $t^{n+1}$ are obtained directly from the (known) values at time $t^n$. For the Laplacian term,  we use a  standard centered finite difference approximation of second order. 
To discretize the fractal operator $\I_\lambda$, we consider the formulation \eqref{liouville}, which  for $\lambda \in ]1,2[$ is a causal variant of \eqref{expression2}. Next, we use a basic quadrature rule to approximate the integral and we use a finite difference approximation of the derivative: 

\begin{equation} \label{discretization1}
\I_{\delta x}[v]_{j}=\delta x^{-\lambda}\,\sum^{+ \infty}_{l=1} l^{1-\lambda}( v_{j-l+1}-2v_{j-l}+v_{j-l-1});
\end{equation}

Note that we have absorbed the constant $\alpha_\lambda$ in $b$.

\begin{remark}
The practical implementation of the schemes requires to make truncations. The main truncation concerns the integral operator for the nonlocal term $\I_{\lambda}$. We replace $\int_{0}^{+\infty}$ with $\int_{0}^L$ and the finite difference approximation becomes:
\begin{equation} \label{discretization2}
\I_{\delta x}[v]_{j}=\delta x^{-\lambda}\,\sum^{A}_{l=1} l^{1-\lambda}( v_{j-l+1}-2v_{j-l}+v_{j-l-1});
\end{equation}
with $L=A \,\delta x$. 
Usually, we consider that functions are 
either compactly supported or constant near $-\infty$ and $+\infty$, so it is legitimate to consider a finite sum for the discretization of $\I_\lambda$. However the truncation parameter $A$ has to be chosen judiciously. Indeed, when $ 0< \lambda < 1$ , the term $l^{1-\lambda}$ in the discretization \eqref{discretization2} is bigger and bigger for increasing  $l$, hence $A$ ought to be big enough. In contrast, whenever $1 < \lambda < 2 $,  $l^{1-\lambda}$ is negligible for large values and is important only for small values of $l$. Thus in this case, it is judicious to take $A$ small enough. We see again here that the non local  effect is all the more important than $\lambda$ is small.
These memory effects strongly depending on $\lambda$ have been reported in \cite{Diet}.
To take this behaviour into account, we   set the truncation parameter $A := \max(100,\frac{10}{\lambda})$.
Of course $100$ is an arbitrary limit to avoid rounding effects for small $\lambda$.

\end{remark}

Since the equation is implemented using an explicit method, this imposes a CFL condition on the time and space steps which ensures the numerical stability that is to say that the difference between the approximate solution and the exact solution remains bounded when $n\rightarrow +\infty$ for $\delta x, \delta t$ given. The stability analysis of the nonlocal scheme
\eqref{explicit} is done in \cite{afaf}. This requires a careful study, owing to the anti-diffusive behaviour of the nonlocal term. Indeed,  the analysis of the PDE \eqref{EDP}  shows that for the continuous problem, low  to medium frequencies are amplified by the nonlocal term, \cite{afaf}.  We come back to this in section \ref{fft}.
Therefore, the standard Von Neumann definition of stability must be slightly modified. Denoting $\theta = k \,\delta x \in (0, 2 \pi)$ and  $u^{n}_j = \exp{(i j \theta)}$,
instead of imposing that the amplification factor $G(\theta) = \frac{u^{n+1}_j}{u^{n}_j} $ must fulfill $|G(\theta)| < 1$  for any frequency
$\theta$,  we only impose  $|G(\theta)| < 1$  for $\theta$ above a given threshold.
We obtain the  two following conditions on the time and space steps:
\begin{equation} 
(1-2^{1-\lambda})\, \frac{b }{\delta x^\lambda} < \frac{2 a}{\delta x^2} ,
\label{scond1}
\end{equation}
\begin{equation} 
\frac{2 a \,\delta t}{\delta x^2} + (2-2^{1-\lambda})\, \frac{b \,\delta t}{\delta x^\lambda} <1.
\label{scond2}
\end{equation}
 The first stability condition \eqref{scond1}  forces the mesh-size $\delta x$ to be small enough in order that the diffusion term should dominate  the nonlocal anti-diffusive term for high frequencies, whereas the second stability condition \eqref{scond2}  looks like an usual CFL condition for explicit schemes and forces the time-step to be small enough.
For more details, we refer the reader to \cite{afaf}. \\

\subsection{Numerical results}

In the following numerical tests, we have to take care to choice $\delta t$ and $\delta x$ accuratly following the conditions \eqref{scond1} and \eqref{scond2}. Thus, the time and space steps depend on the choice of $a, \, b$ and $\lambda$. 
We begin by consider an electrocardiogram (ECG) signal.
Figure \ref{DF1} illustrates a comparison between the PDE \eqref{EDP} implemented using a finite difference method (FDM) and the Savitzky-Golay smoothing filter. \\
The Savitzky-Golay smoothing filter also called digital smoothing polynomial filter or least-squares smoothing filter was first described in 1964 by Abraham Savitzky and Marcel J. E. Golay. \cite{sgolay}.
Essentially, the Savitzky-Golay method performs a local polynomial regression on a series of values that is to say it replaces each value of the series with a new value which is obtained from a polynomial fit to neighboring points. 
The algorithm is based on the following equation
\begin{equation*}
u_j=\left[ a_0 v_j+ \sum_{i=1}^{\frac{n+1}{2}} a_i(v_{j-i}+v_{j+i}) \right] 
\end{equation*}
where $n$ is the number of data points, $a_i$ are (positive) constants, $v$ is the noisy signal and $u$ defines the filtered signal.
The main advantage of this approach is the preservation of features of the signal such as relative maxima, minima and width. Indeed, usually these characteristics  are 'flattened' by classical averaging techniques such as moving averages. 
Figure \ref{DF1} shows three plots: The noisy ECG signal, the smoothed signal (red) using the PDE filter \eqref{EDP} superimposed with
the noiseless signal (blue) and the smoothed signal (red) using the Savitzky-Golay filter superimposed with
the noiseless signal (blue). As we can remark, for these parameters, the low/medium frequencies are not amplified but the relative maxima and minima are preserved. The denoising seems correct and the filtering with our numerical scheme and with Savitzky-Golay seems similar, see better. 
We evaluate our approach and we compare it with Savitzky-Golay in section \ref{metrique}.
%\begin{figure}[!ht!] 
%\begin{center}
%\includegraphics[scale=0.8]{DF1.eps}
%\caption{Three plots are: The noisy ECG signal, the smoothed signal using the PDEs-filter based on the finite difference scheme \eqref{explicit} with $a=0.5, \, b=0.3, \, \lambda=\frac{4}{3}$ and the noiseless signal  \label{DF1}}
%\end{center}
%\end{figure}

\begin{figure}[!ht!] 
\begin{center}
\includegraphics[scale=0.5]{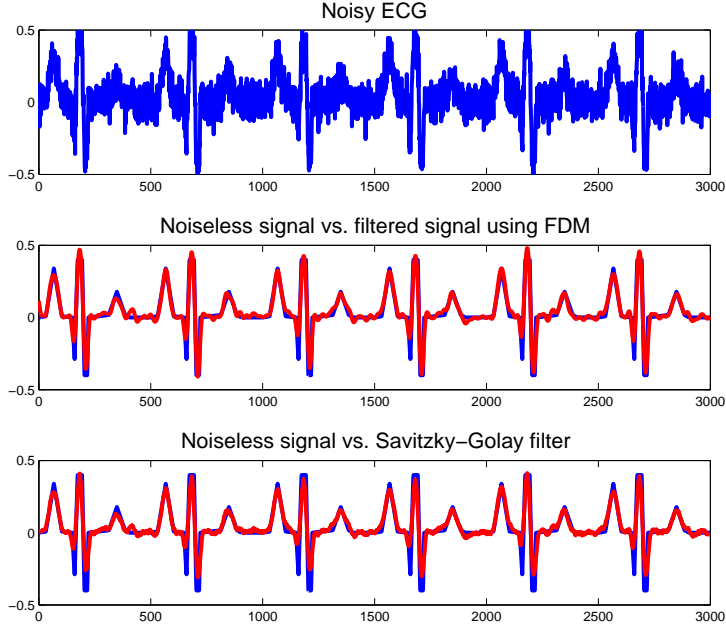}
\caption{ Top: a noisy electrocardiogram $u$; Middle: in red the noiseless signal,  in blue  the filtered signal  with finite difference scheme \eqref{explicit} with  $a=4, b=0.5 , \lambda=1.7, \delta x=1$ and $\delta t=0.1$; Bottom : Noiseless signal (blue) vs. filtered signal using Savitzky-Golay filter. \label{DF1}}
\end{center}
\end{figure}

Figure \ref{justdenoisingDF} illustrates two filtered signals of different types. We start from the simplest possible example of nonlocal filter denoising applied to a signal $u(x)$ consisting of a piecewise constant step function $v(x)=-1$ for $0<x<1$ and $v(x)=1$ for $1<x<2$ corrupted by additive Gaussian white noise $n(x)$ with  standard deviation $\sigma=0.4$:
\begin{equation}
u(x)=v(x)+n(x).
\end{equation}
The result of the denoising is illustrated in figure \ref{indicatrice}. As we can notice, the noise is very well eliminated and we find again our original signal, the shape of the signal has been preserved. The result is better than filtering Savitzky-Golay. Moreover, unlike Discrete Fourier Transform where the contrast is low in the neighbourhood of the discontinuity (see figure \ref{indicatricefft}), the dicontinuity/shock is conserved. 
Therefore, the finite difference method is more suitable for this kind of signal. 
Figure \ref{DFsmtlb} shows another example of good denoising. 

\begin{figure}[h!]
	\centering
	    \subfigure[ $a=3.5, \, b=0.2, \, \lambda=0.1, \delta t=10^{-6}$ and  \hspace{0.2 cm} $\delta x=0.001$.]
	{\includegraphics[scale=0.3]{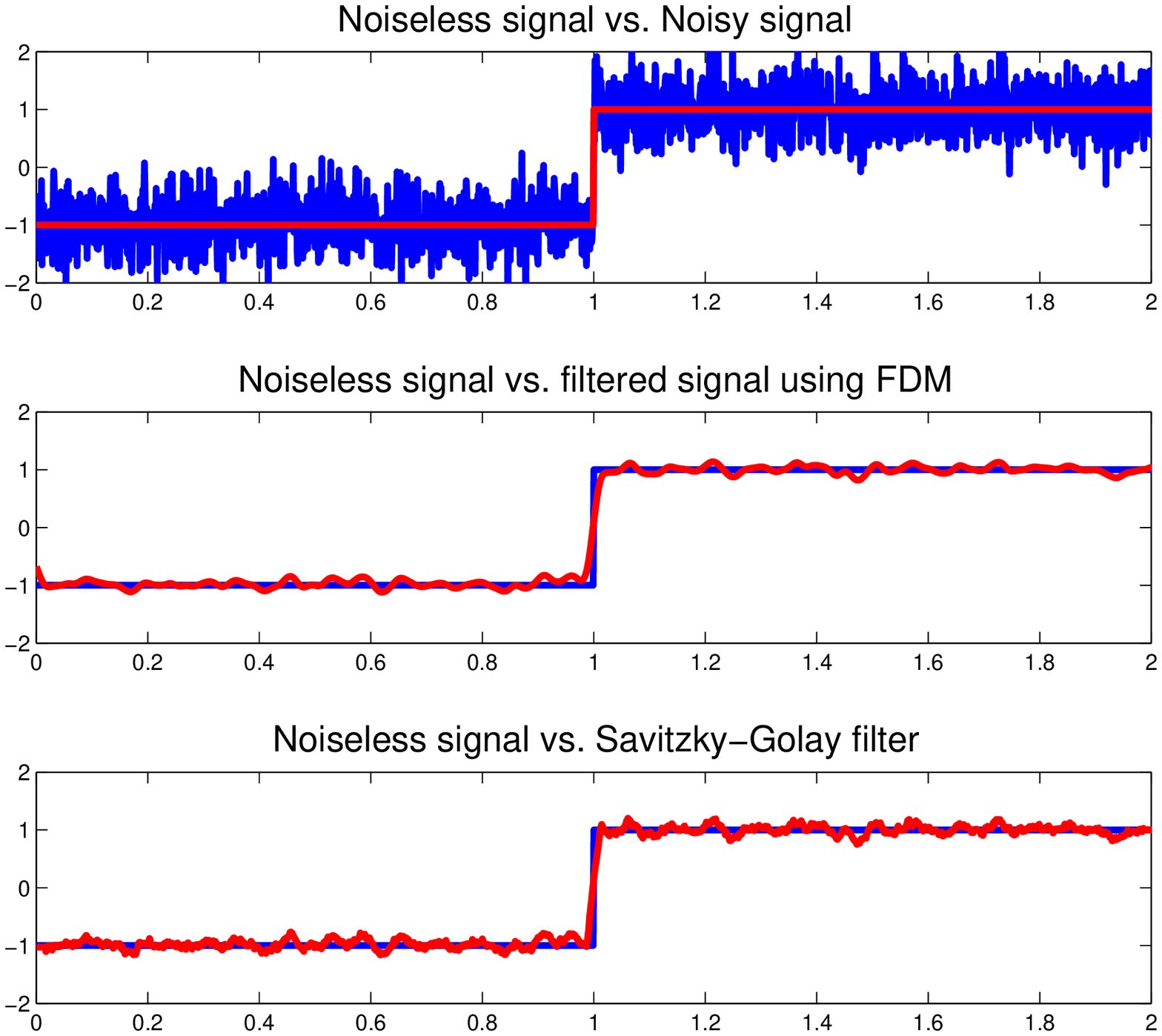}  \label{indicatrice} }
        \subfigure[$a=3.5, \, b=0.3, \, \lambda=1.1, \delta t=0.1$ and $\delta x=1$.  ]
	{\includegraphics[scale=0.3]{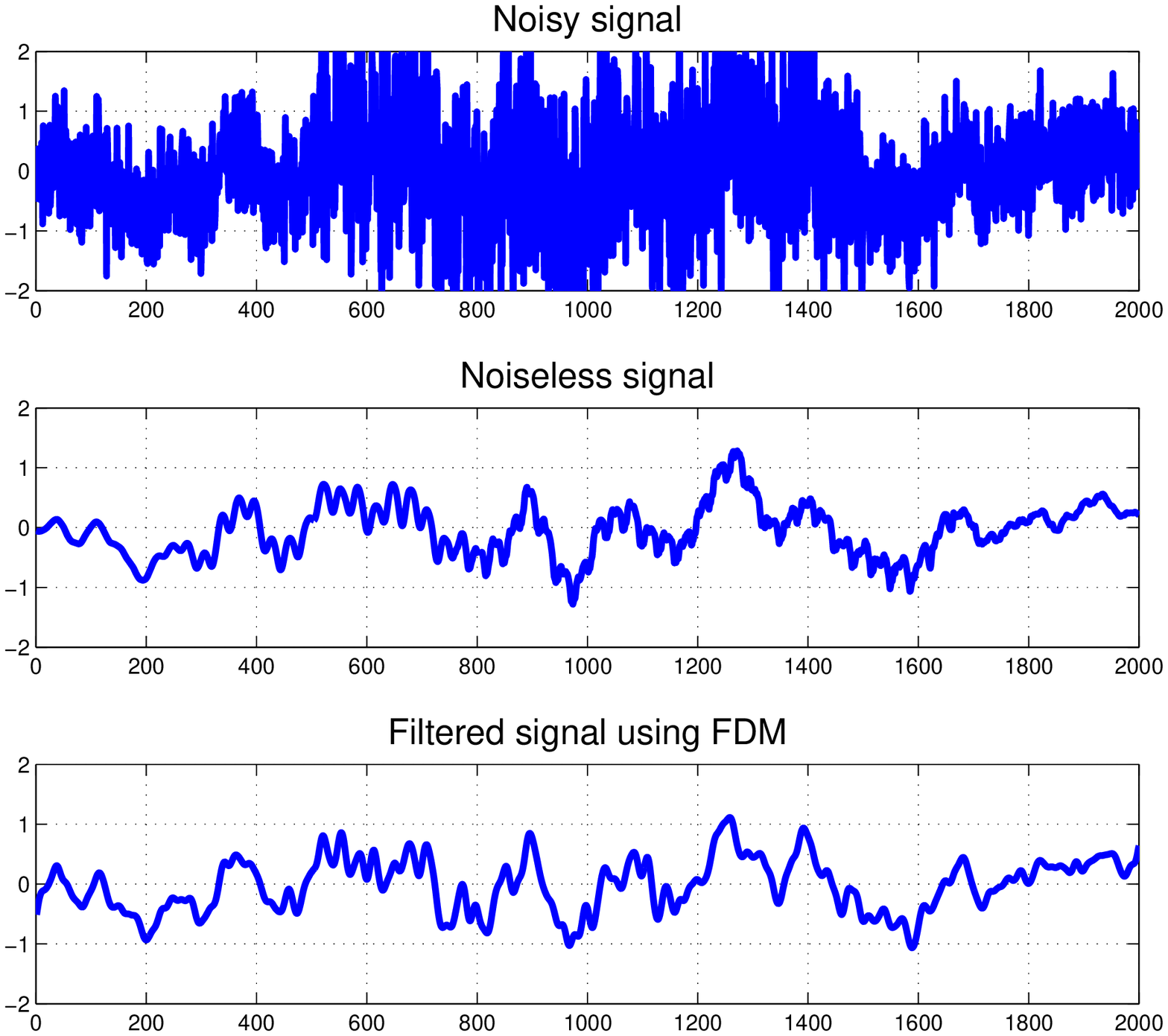} \label{DFsmtlb} }
        \caption{ 
        (a) Top: Noisy signal (blue) vs. noiseless signal (red); Middle: Noiseless signal (blue) vs. filtered signal using FDM (red); Bottom: Noiseless signal (blue) vs. filtered signal using Savitzky-Golay filter (red). (b) Top: Noisy signal; Middle: Noiseless signal ; Bottom: filtered signal using FDM .
        \label{justdenoisingDF} }
\end{figure} 

%\begin{figure}[!ht!] 
%\begin{center}
%\includegraphics[scale=0.75]{smtlbmieux2.eps}
%\caption{Top: a noisy signal. Middle: the noiseless signal. Bottom: the filtered signal with finite difference scheme  \eqref{explicit} 
%with $a=3.5, \, b=0.3, \, \lambda=1.1, \delta t=0.1$ and $\delta x=1$.  \label{DFsmtlb1}}
%\end{center}
%\end{figure}

\section{Numerical results using Discrete Fourier Transform \label{fft}}

In this section, we first discuss the choice of parameters $a,b,\lambda$ and next, we give some numerical results based on fast Fourier transform (FFT) which is an efficient algorithm to compute the discrete Fourier transform.  

\subsection{Choice of parameters \label{choix} }

We fix $T=1$. So, we can rewrite the kernel in Fourier variable as follows:
\begin{equation*}
K_{a,b}^{\lambda}(\xi)=e^{-\psi(\xi)}
\end{equation*}
where $ \psi(\xi) = 4 \pi^2 a\xi^2-b|\xi|^{\lambda}$.
\begin{figure}[h!]
	\centering
        \subfigure[ $4 \pi^2  a = 0.01 $ and $ \lambda = 1.5$]
	{\includegraphics[scale=0.5]{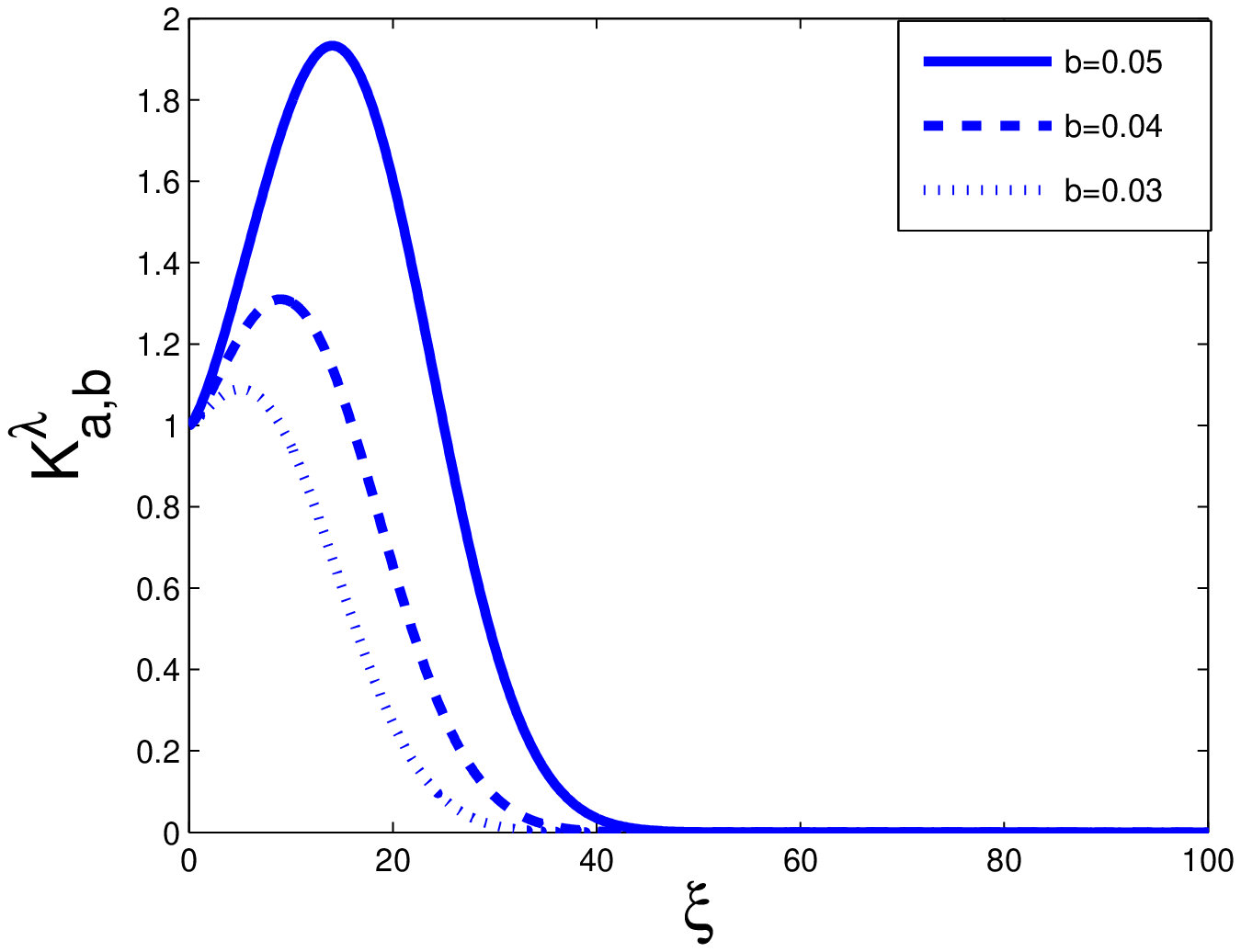}}
	\subfigure[ $4 \pi^2 a = 0.01 $, $ b =0.05 $ and $\lambda=1.5$ ]
	{\includegraphics[scale=0.35]{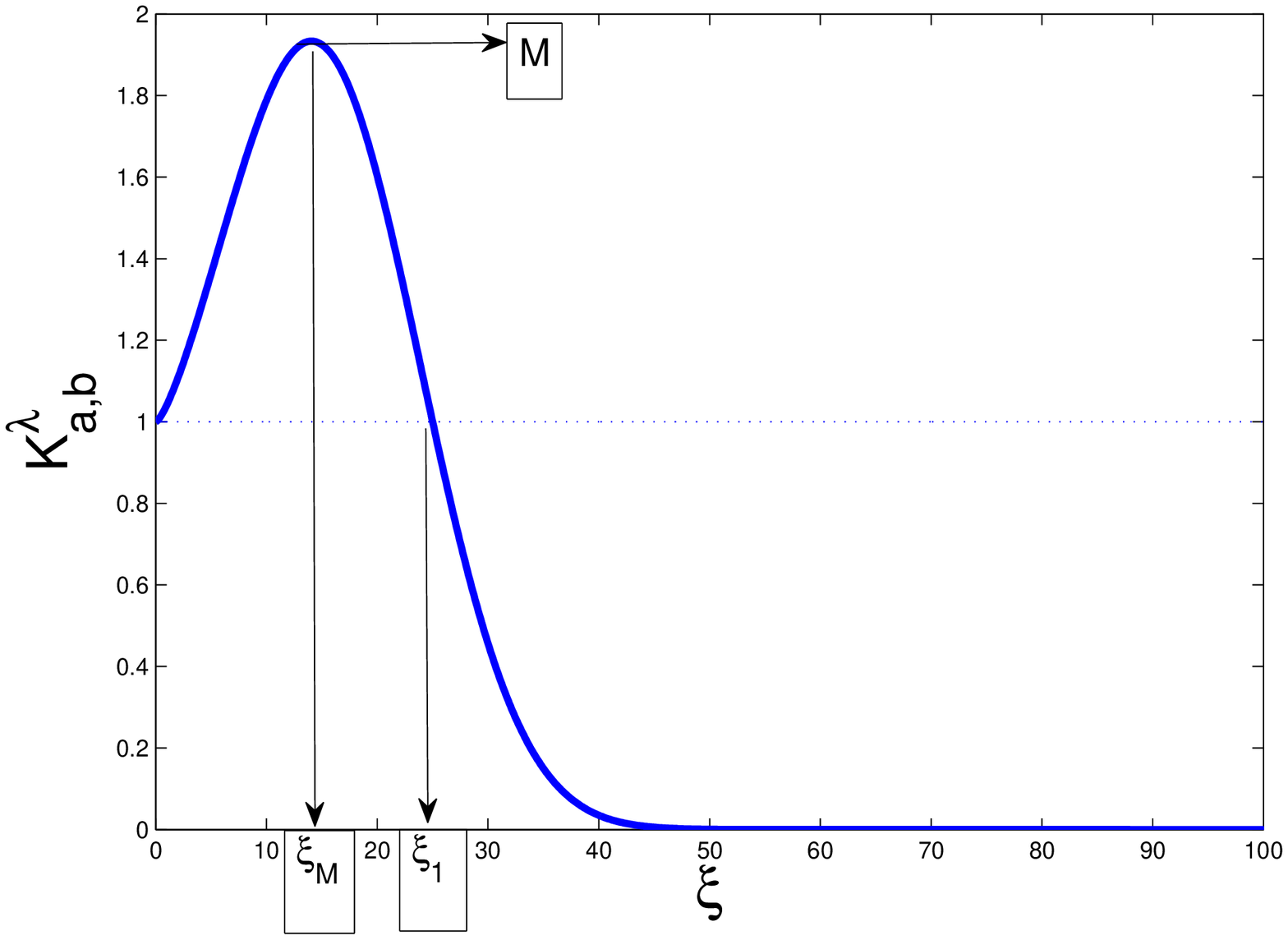}}
        \caption{ Evolution of the spatial Fourier transform of the kernel $K$ for different parameters.\label{fenetre}
}
\end{figure} 
We draw in figure \ref{fenetre} the behaviour of the kernel 
$K_{a,b}^{\lambda}$ for $a,\lambda$ fixed and for different values of $b$. 
As $K_{a,b}^{\lambda}$ reaches it maximum at $\xi_M:=\left(\frac{\lambda b}{8 \pi^2 a}\right)^{\frac{1}{2-\lambda}}$ and $K_{a,b}^{\lambda}(\xi_M)=e^{4 \pi^2 a(\frac{\lambda b}{8 \pi^2 a})^{\frac{2}{2-\lambda}}(\frac{2}{\lambda} - 1)} $, we deduce that $K_{a,b}^{\lambda}(\xi_M)>1$ if and only if $0<\lambda<2$. Therefore, whatever the choice of $\lambda \in ]0,2[$ , we always have an amplification of medium frequencies. Of course, the magnitude of amplification depends on $a,b,\lambda$. 
Another natural frequency is  $\xi_1=\left(\frac{b}{4 \pi^2 a}\right)^{\frac{1}{2-\lambda}} $ which represents the  neutral frequency and satisfies $K^{\lambda}_{a,b}(\xi_1)=1$ and $K^{\lambda}_{a,b}(\xi_1) < 1$ for $\xi > \xi_1$. Therefore $\xi_1$ is the threshold above which dampening will occur. 

Now, we wish to control both the denoising and contrast enhancement. For that, we follow a simple strategy: 

\begin{enumerate}
\item In a first step, we wish to control the frequency range to amplify. This one can be controlled by the following ratio: $ \frac{\xi_1}{\xi_{M}} = (\frac{2}{\lambda})^{1/2-\lambda} $. Indeed, 
this ratio is decreasing w.r.t. parameter  $\lambda$ and shrinks from $+\infty$ to $\sqrt{e}$ when 
$\lambda$ varies from $0$ to $2$.  Hence, to choose a given amplified frequency range, it is enough to fix  parameter $\lambda$ and  the ratio
 $\frac{b}{a}$  (which in turn determines
 $\xi_1=\left(\frac{b}{4 \pi^2 a}\right)^{\frac{1}{2-\lambda}}$). 

\item In this step, we want to monitor the denoising, which will start for frequencies above the neutral  one $\xi_1$. One can easily check the following equality:
 $$ \psi(\alpha \xi_{1}) = 4 \pi^2 a \,\xi_1^2 (\alpha^2-\alpha^{\lambda})  $$
 where $\alpha$ is a any positive constant. Therefore, the greater $a$ is, the smaller $K^{\lambda}_{a,b}(\alpha \xi_1) $ will be, which means that the dampening rate will increase. Hence parameter $a$ monitors the denoising intensity, this is coherent because it controls the Laplacian term. 
 
\item   To finish with, we wish to monitor the value $M:=K^{\lambda}_{a,b}(\xi_M)$ in order to control the contrast enhancement.
Indeed, the higher $M$ is, the better the constrast will be amplified.
Parameters $\lambda$ and  $\frac{b}{a}$  being fixed by item 1  it is enough to adjust both parameters $b$ and $a$  while keeping the ratio $\frac{b}{a}$ within some bounds in order to preserve the amplified  frequency range. The expression of $M$ shows that to have a good amplification of medium frequencies, it is necessary to increase $b$ as well as  $a$. This is quite natural since the coefficient $b$ controls the anti-diffusive term and has an opposite behaviour to the Laplacian term which tends to flatten the signal. In our method, contrast and denoising are no more antagonistic. 
We come back on the tuning of parameter $b$ in section \ref{metrique}.
\end{enumerate}

\subsection{Denoising}
In this subsection, we are only going to highlight the ability of noise reduction of our nonlocal equation \eqref{EDP}.\\ 
 To begin with, we consider an ECG signal to illustrate the denoising using the FFT to solve the fractal equation \eqref{EDP}. The result is given in the figure \ref{ecgjustdenoising}.   
\begin{figure}[h!]
	\centering
        \subfigure[ $4 \pi^2 a = 0.01 , b = 0.03$ and $ \lambda = 0.5$ ]
	{\includegraphics[scale=0.4]{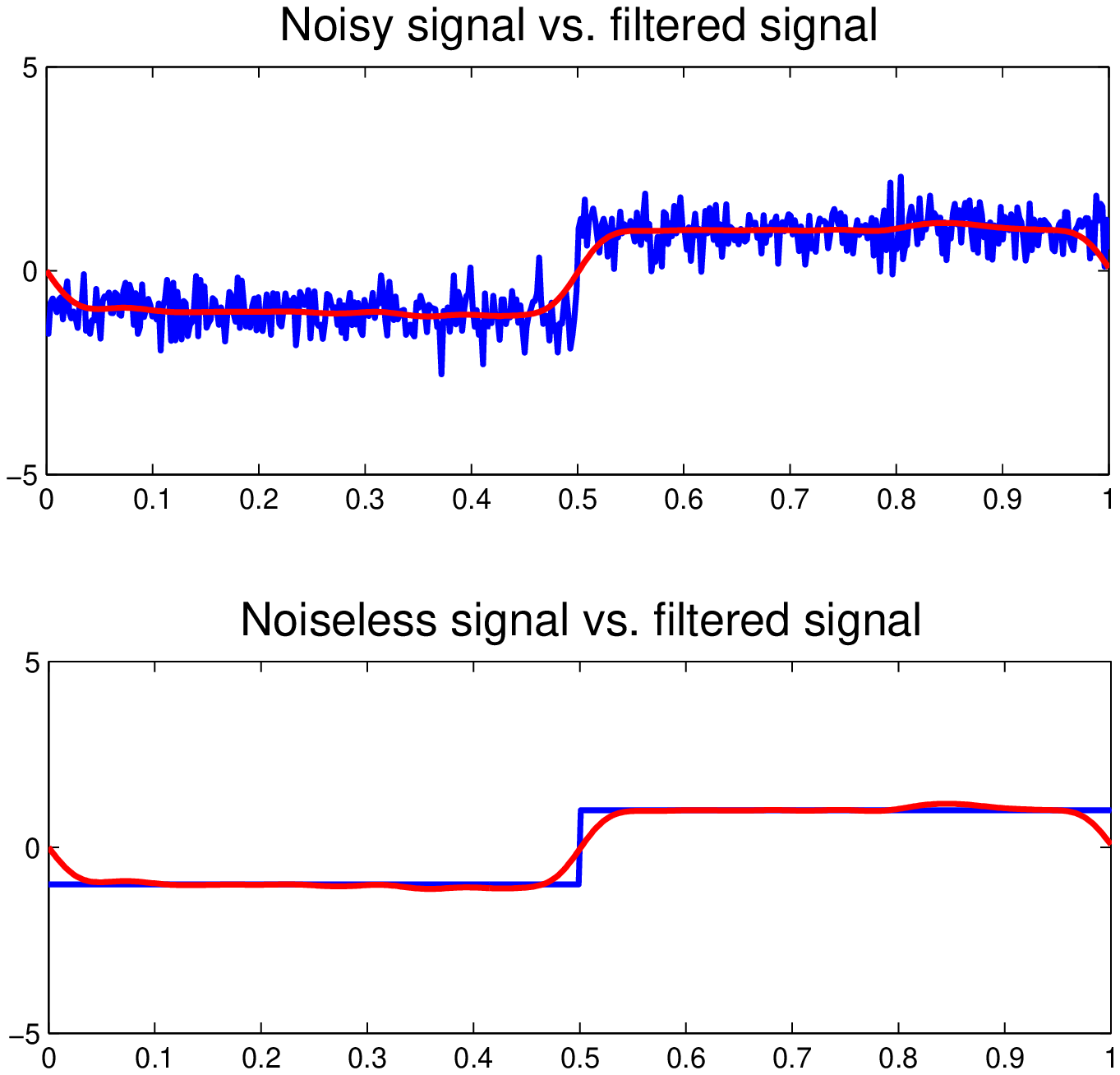} \label{indicatricefft} }
	\subfigure[ $ 4 \pi^2 a =0.005 1 , b = 0.015 $ and $ \lambda = 1.7$]
	{\includegraphics[scale=0.4]{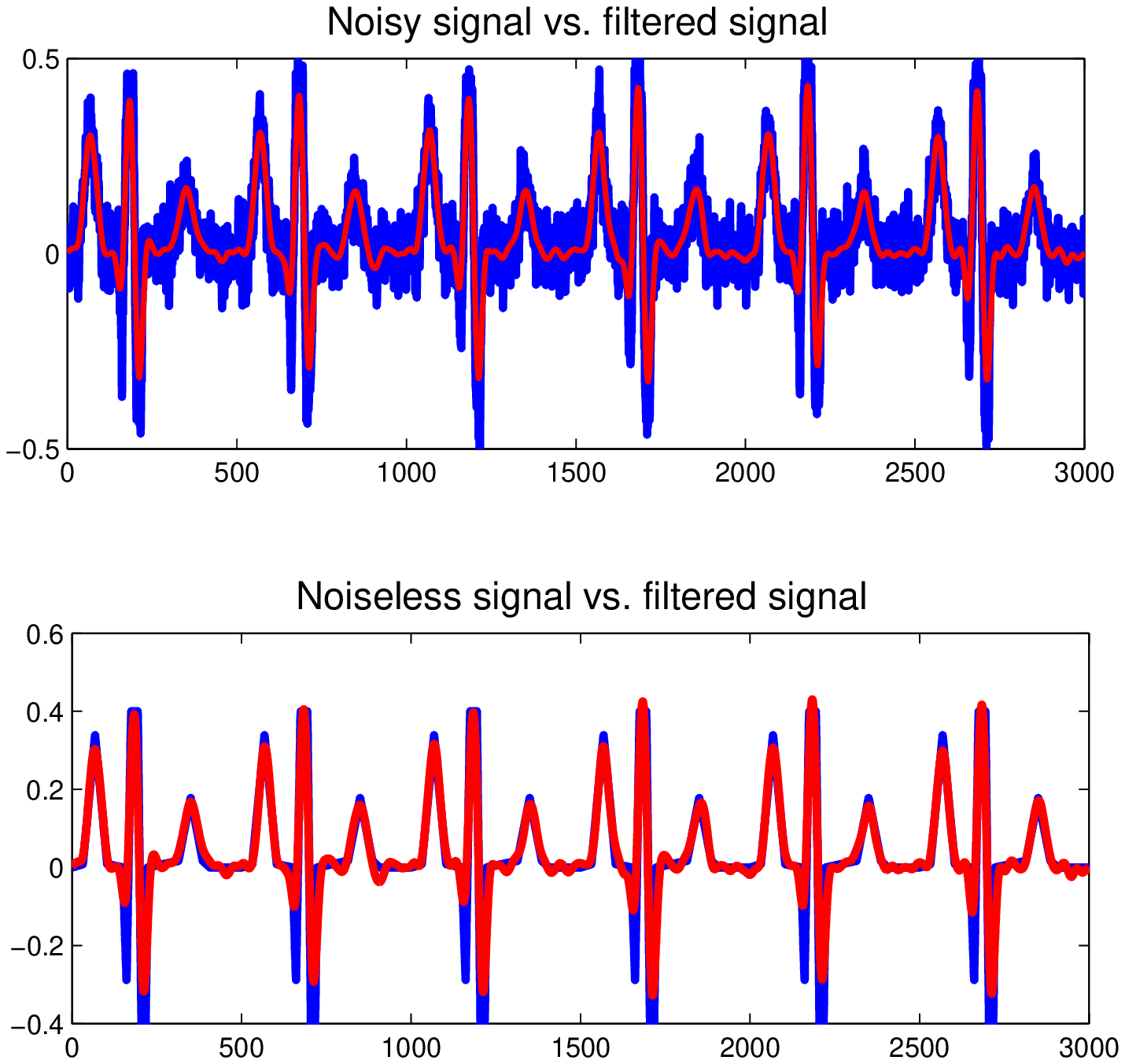}  \label{ecgjustdenoising} }
        \caption{ Top: Noisy signal (blue) vs. filtered signal (red). Bottom: noiseless signal (blue) vs. filtered signal (red).  \label{justdenoising} }
\end{figure} 
For this choice of parameters, we note that the denoising is suitable and that the relative maxima are preserved. However, the relative minima are not completely preserved, in fig. \ref{ecgjustdenoising} bottom. 
%Whereas for the Savitzky-Golay filter, we have both preservation of relative minima and maxima. 
At last, Figure \ref{sinusoidalb} describes the denoising of a sinusoidal signal. We can see the performance of the denoising of nonlocal approach discretized with FFT, which works very well for this type of periodic signal.   

\subsection{Denoising \& Enhancement of signals \label{both}}
In this part, we are going to highlight both the ability of noise reduction and contrast enhancement. Of course, to emphasize the amplification, we take into account the remarks done in subsection \ref{choix}.  

We start from a simple example of simulatenous denoising and enhancement applied to a strongly attenuated sinusoidal signal highly corrupted by a random noise.
In figure  \ref{fourier125} middle, we plot the original noiseless signal, the same signal amplified $\times 50$ and the filtered signal, performed with our non local FFT method.  As we can notice in figure \ref{fourier125}, the noise is eliminated and the contrast is well  amplified.  

\begin{figure}[!ht!] 
\begin{center}
\includegraphics[scale=0.5]{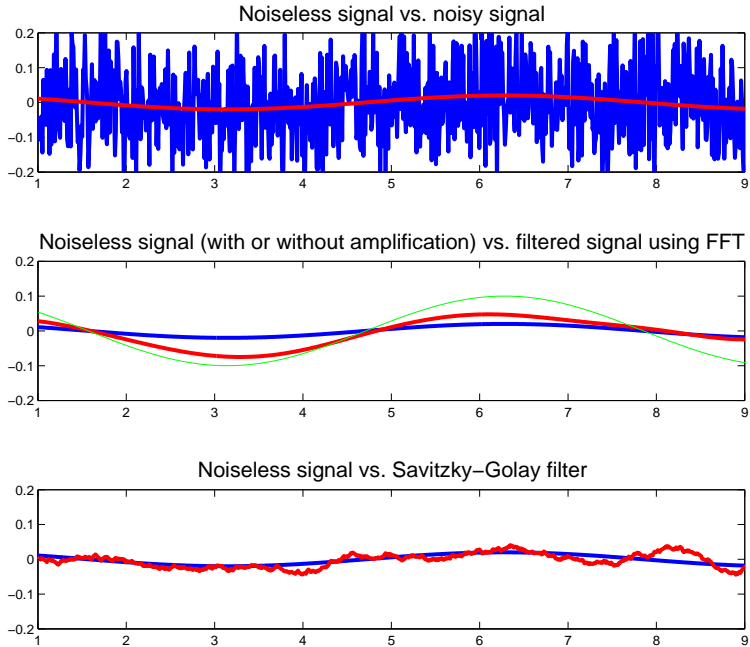}
\caption{Top: Noisy signal (blue) vs. noiseless signal $v(x)=0.02 \cos (x)$ (red); Middle: Noiseless signal (blue) vs. filtered signal (red) using FFT with $ 4 \pi^2 a = 0.2, b = 1.25, \lambda = 0.5$ superimposed with \emph{amplified} noiseless signal $cos(x)$ (green);  Bottom: Noiseless signal (blue) vs. filtered signal (red) using Savitzky-Golay method.
 \label{fourier125}}
\end{center}
\end{figure}

\begin{figure}[!ht!] 
\begin{center}
\includegraphics[scale=0.7]{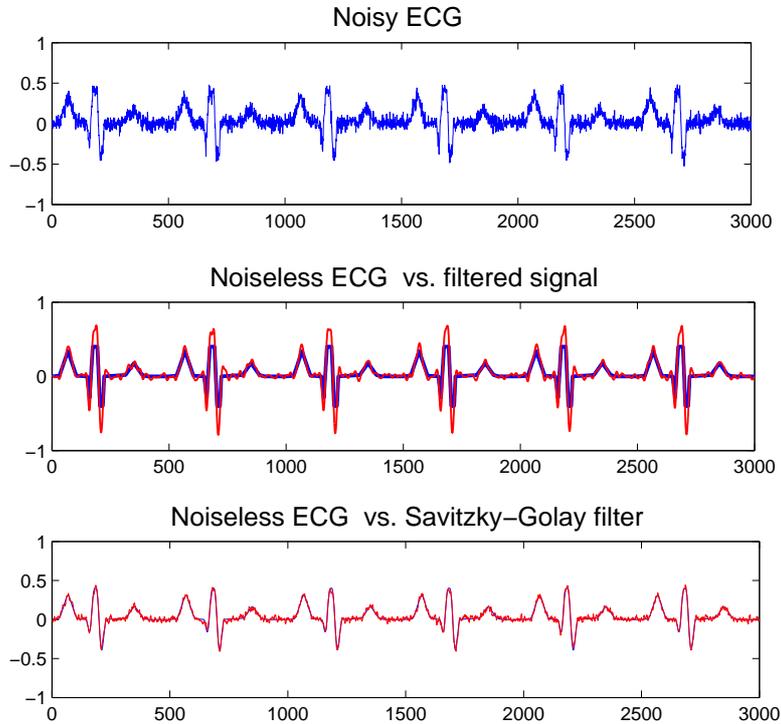}
\caption{Top: a noisy electrocardiogram $u(x)$; Middle: the noiseless signal $v(x)$ (blue) and denoising \& enhancement of the ECG signal (red) using the FFT on the fractal equation \eqref{EDP} with $ 4 \pi^2 a = 0.005, b = 0.0225$ and $\lambda = 1.7$. Bottom: filtering using Savitzky-Golay method.
 \label{ecgfft}}
\end{center}
\end{figure}

Figure \ref{ecgfft} illustrates the smoothing of an electrocardiogram signal by filtering the noise with Savitzky-Golay filter (third plot) and by denoising and enhancement  with the fractal scalar conservation law \eqref{EDP} (second plot).
In the second case, we can see that the relative maxima and minima are amplified. 
Of course, we can obtain even greater amplification of low/medium frequencies by suitably tuning the parameters but, in these conditions, we will obtain sizeable variations between each peak. Thus,   when we wish to amplify the low frequencies, we have to be careful  that we do not alter too much the shape of the noiseless signal.    

To rate the performance of denoising and enhancement of our filter, we consider signals with different SNR, which corresponds to the power ratio between a signal and the background noise. 
\begin{figure}[!ht!] 
\begin{center}
\includegraphics[scale=0.7]{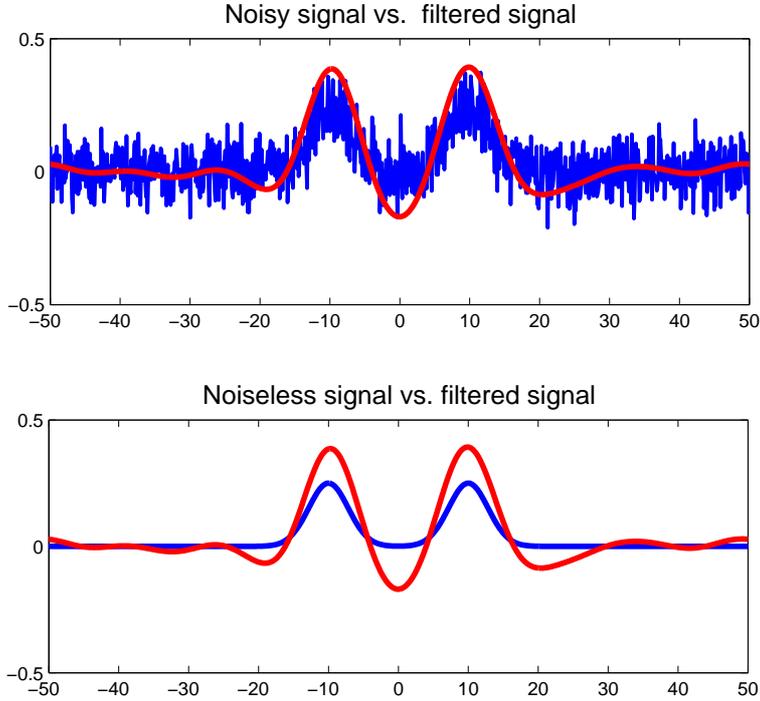}
\caption{First plot (top) represents a noisy signal $u(x)$ (blue). Second plot (bottom) illustrates  the noiseless signal $v(x)$ (blue) and denoising \& enhancement of the signal (red) using the FFT on the fractal equation \eqref{EDP} with $4 \pi^2 a = 0.1, b =0. 3$ and  $\lambda = 1.5$.
 \label{dunes2}}
\end{center}
\end{figure}

\begin{figure}[!ht!] 
\begin{center}
\includegraphics[scale=0.5]{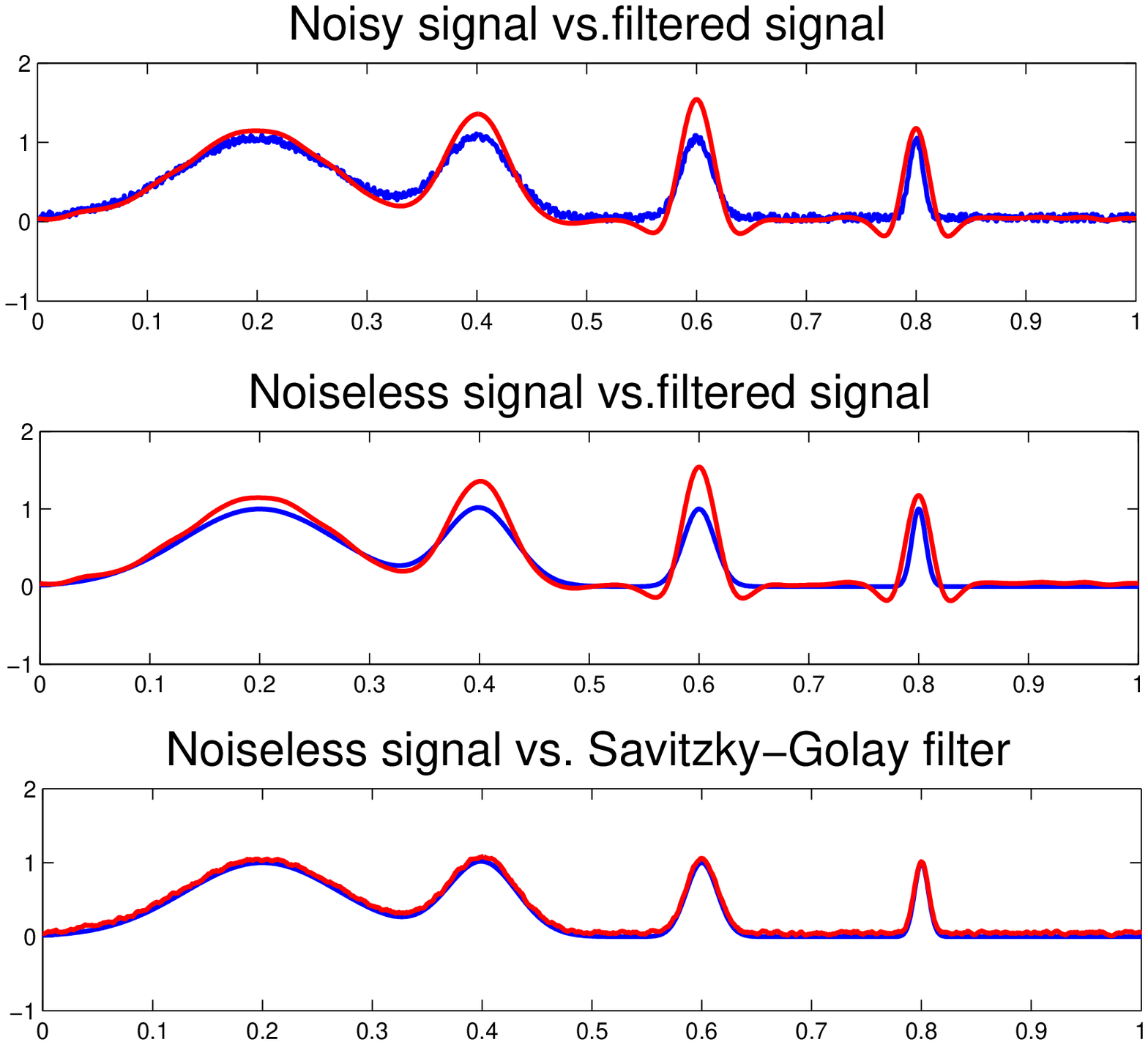}
\caption{ Choice of parameters: $ 4 \pi^2 a =0.0 1, b = 0.05$ and  $\lambda = 1.5$. The red line depicts the filtered signal. 
 \label{dunesbcp}}
\end{center}
\end{figure}

In figure \ref{dunes2}, we took a small SNR.
We can observe that shapes are amplified and that  noise has been reduced significantly. But, we can also remark  the undershoots just before and behind the shapes. This phenomenon has been highlighted in \cite{alibaud} in the setting of dunes morphodynamics. It is a consequence of the mass conservation property, see equation \eqref{conserve}. We obtain similar results in figure \ref{dunesbcp} with a large SNR. Indeed, as we can see, the proposed method of filtering  yields \emph{both} an interesting denoising and an amplification of low/medium frequencies. The third plot conveys the filtering using the Savitzky-Golay approach. As agreed upon, we obtain a preservation of relative maxima. Moreover, comparing the output obtained thanks to our nonlocal method with the one of Savitzky-Golay filter, we notice the better ability to \emph{increase} local extrema (contrast enhancement) while keeping a good denoising. 
% Of course, we do not find exactly the initial signal $v$ but the final signal obtained by filtration is more or less similar to the signal $v$.
\begin{figure}[h!]
	\centering
	\subfigure[ $ 4 \pi^2 a =0.01 , b =0.03$ and $\lambda = 1.5$ ]
	{\includegraphics[scale=0.45]{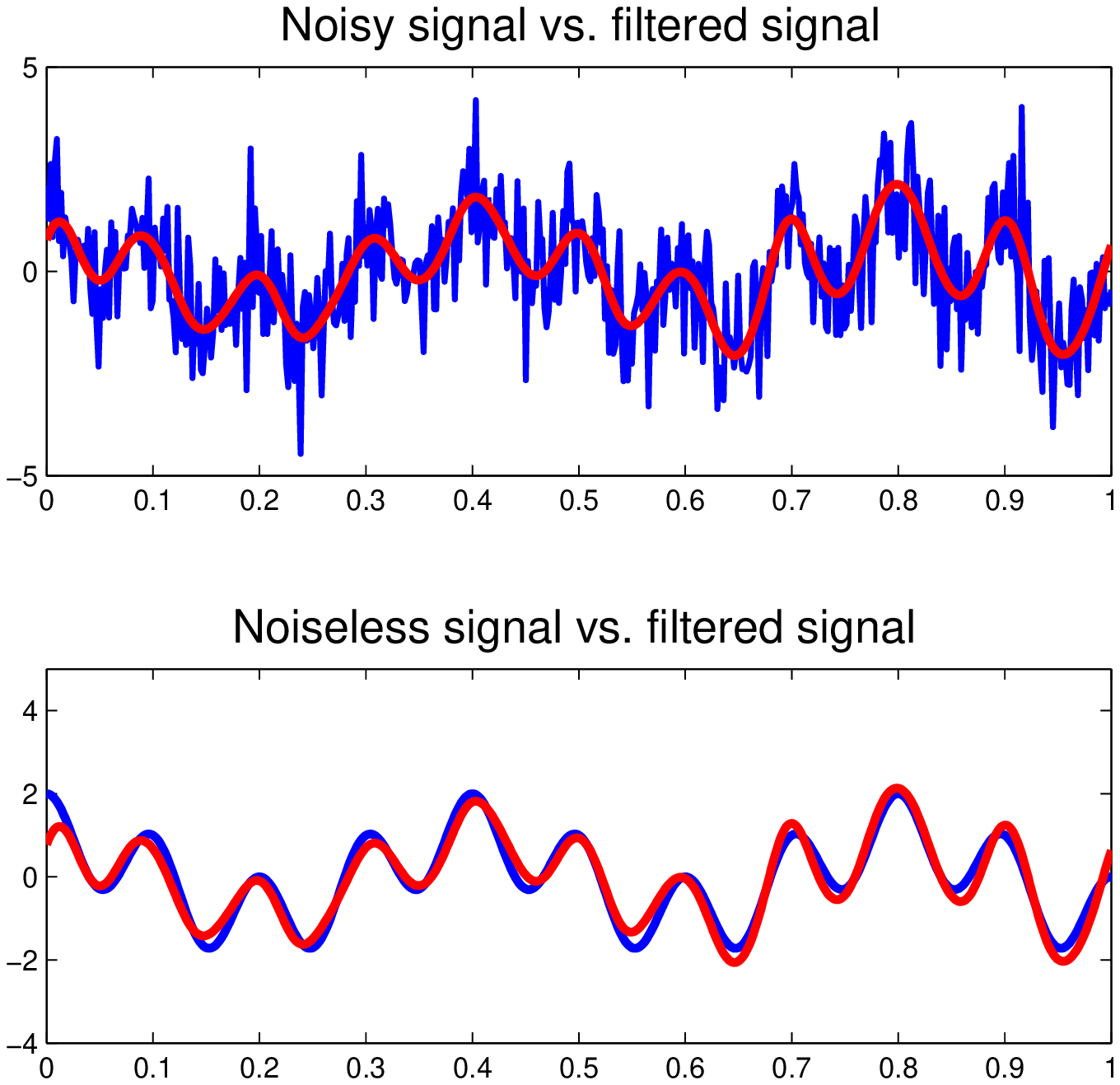}  \label{sinusoidalb} }
	 \subfigure[ $4 \pi^2 a =0.01 , b =0.05$ and $\lambda = 1.5$]
	{\includegraphics[scale=0.45]{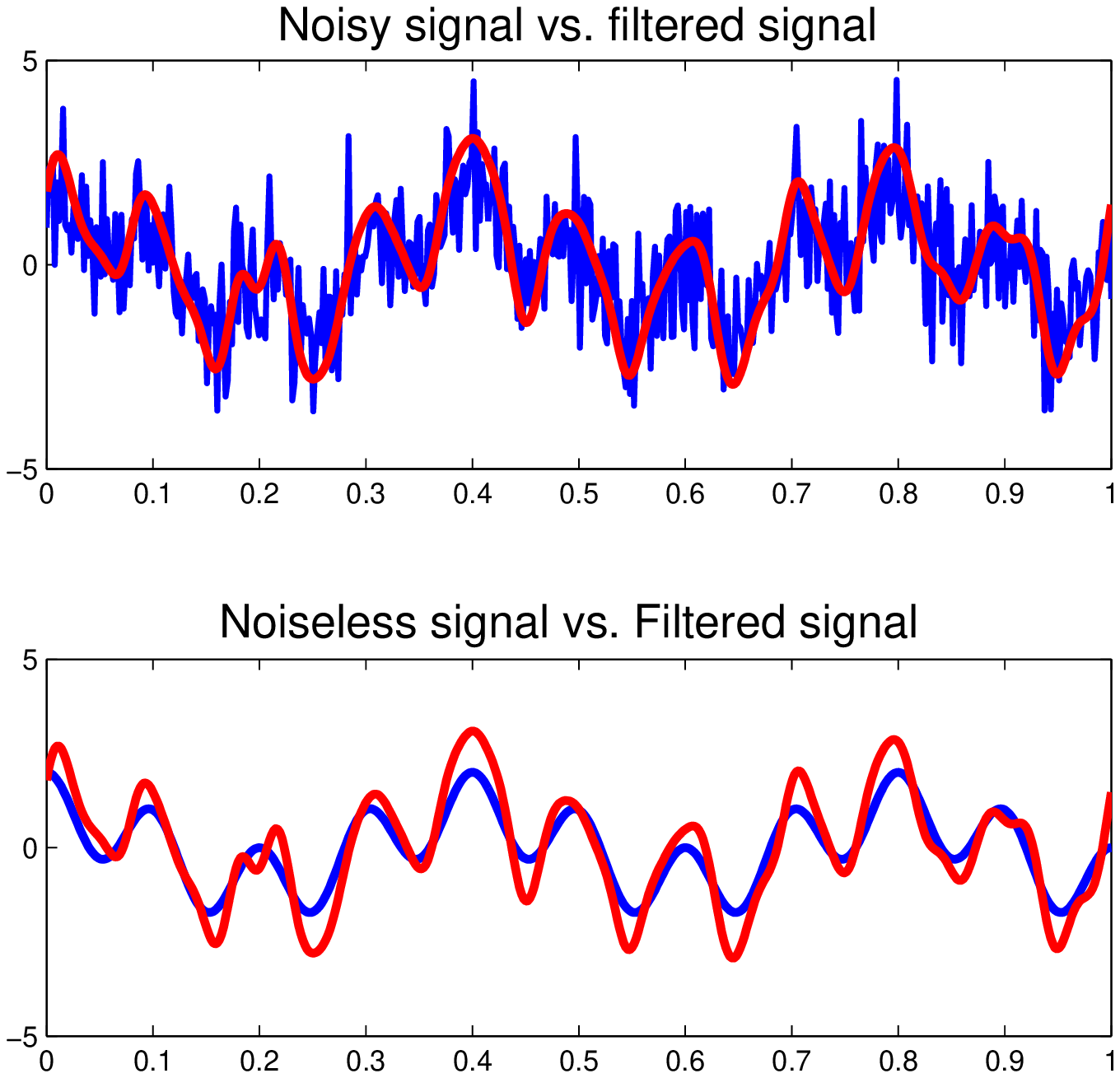} \label{sinusoidala}}
        \caption{ The red line depicts the filtered signal. In this case, we took $u_0(x) = \cos(5\pi x) + \cos(20 \pi x)$. \label{sinusoidale}
}
\end{figure} 
This statement is confirmed by Figure \ref{sinusoidala}, where the  signal has a medium SNR.
We obtain a ``good'' smoothing and an interesting amplification of the low and medium frequencies.

\section{Performance metrics \label{metrique}}

In this section, we wish to measure the denoising ability of our model.  
To evaluate our approach and compare it with the Savitzky-Golay filter, we use two measures: the Mean Square Error (MSE) and the Signal-to-Noise Ratio (SNR). These metrics are frequently used in signal processing. They are defined as follows:
\begin{equation*}
MSE=\frac{1}{N} \sum_{i=1}^{N} (u_0(i)-u(i))^2
\end{equation*}
\begin{equation*}
SNR_{db}=10 \log_{10} \left( \frac{\sum_{i=1}^{N} u_0(i)^2}{\sum_{i=1}^{N} (u_0(i)-u(i))^2}  \right)
\end{equation*}
where $u_0$ is the   noiseless original signal, $u$ is the filtered signal and $N$ is the length of the filter. \\ 
It is easy to see that a small MSE corresponds to a high noise reduction and that a large SNR indicates a good denoising. 
%value represents the better noise reduction and that the larger SNR value indicates that the denoising of signal is closer to the signal. \\
To compare the performance of filters, we consider a signal of trigonometric type and an ECG signal. Noise is added to these signals with SNR varying between 0 to 8. Results  are plotted in figures \ref{sinusoidalb} and \ref{ecgjustdenoising}.
For each signal, we  use a sample of 100 random noises. 
The performance of the two approaches is estimated using MSE and SNR criteria. We plot an average of  the results on  figures \ref{snrmsesinus} and \ref{snrmseecg}. Figures \ref{sinusoidalea} and \ref{ecga} show SNR values for these two methods applied to trigonometric and electrocardiogram signals. Figures \ref{sinusoidaleb} and  \ref{ecgb} show MSE values.
These values show  the high performance of nonlocal approach in signal denoising for trigonometric signals.
Remember that in this case, the algorithm used for the implementation of the equation is the FFT, which is most suitable for trigonometric  signals. We may just notice in figure \ref{snrmseecg} that for high SNR, the  Savitzky-Golay filter is better than the nonlocal filter approach. Still, when the SNR is low - and it is the critical case - the proposed method is  more efficient than the Savitzky-Golay.  Thus, the implementation of our PDE based on the FFT may not be effective for any type of signal, at least when the SNR is high. 
%Therefore, to filter the ECG signal, we also tried  the finite difference scheme. 
Alternatively, we may use the finite difference scheme.
Results  of the filtering are illustrated in figure \ref{snrmseecgfd}. As we can see, the proposed model is always more efficient than the Savitzky-Golay approach, but for low SNR  the FFT scheme remains best. Finally, regardless the type of signal considered, the fractal conservation law \eqref{EDP} is a good tool of denoising, provided that  it is implemented with the right method.        

\begin{figure}[h!]
	\centering
        \subfigure[ SNR values in Savitzky-Golay method in comparison to proposed method ]
	{\includegraphics[scale=0.3]{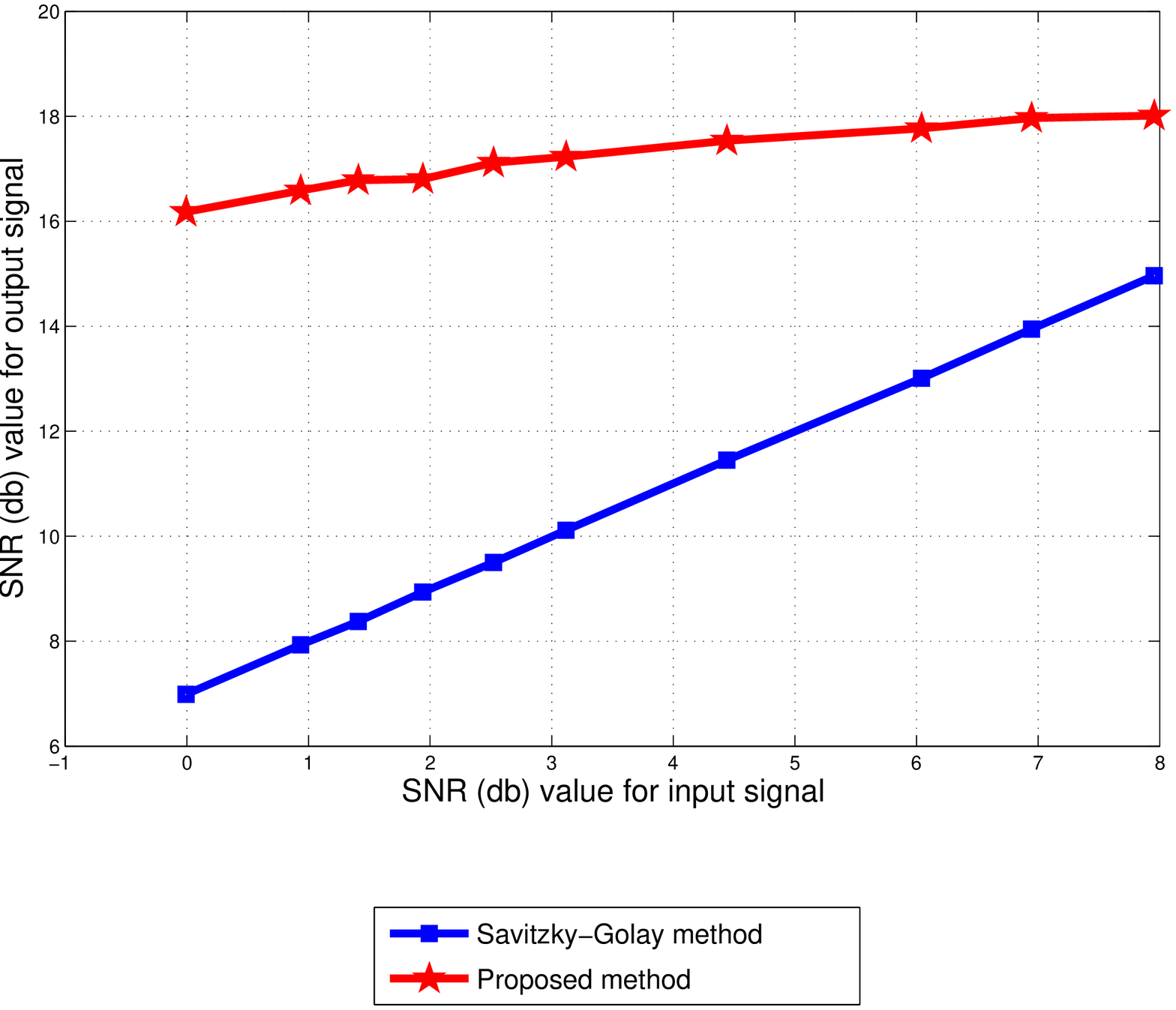}  \label{sinusoidalea} }
	\subfigure[ MSE values in Savitzky-Golay method in comparison to proposed method ]
	{\includegraphics[scale=0.3]{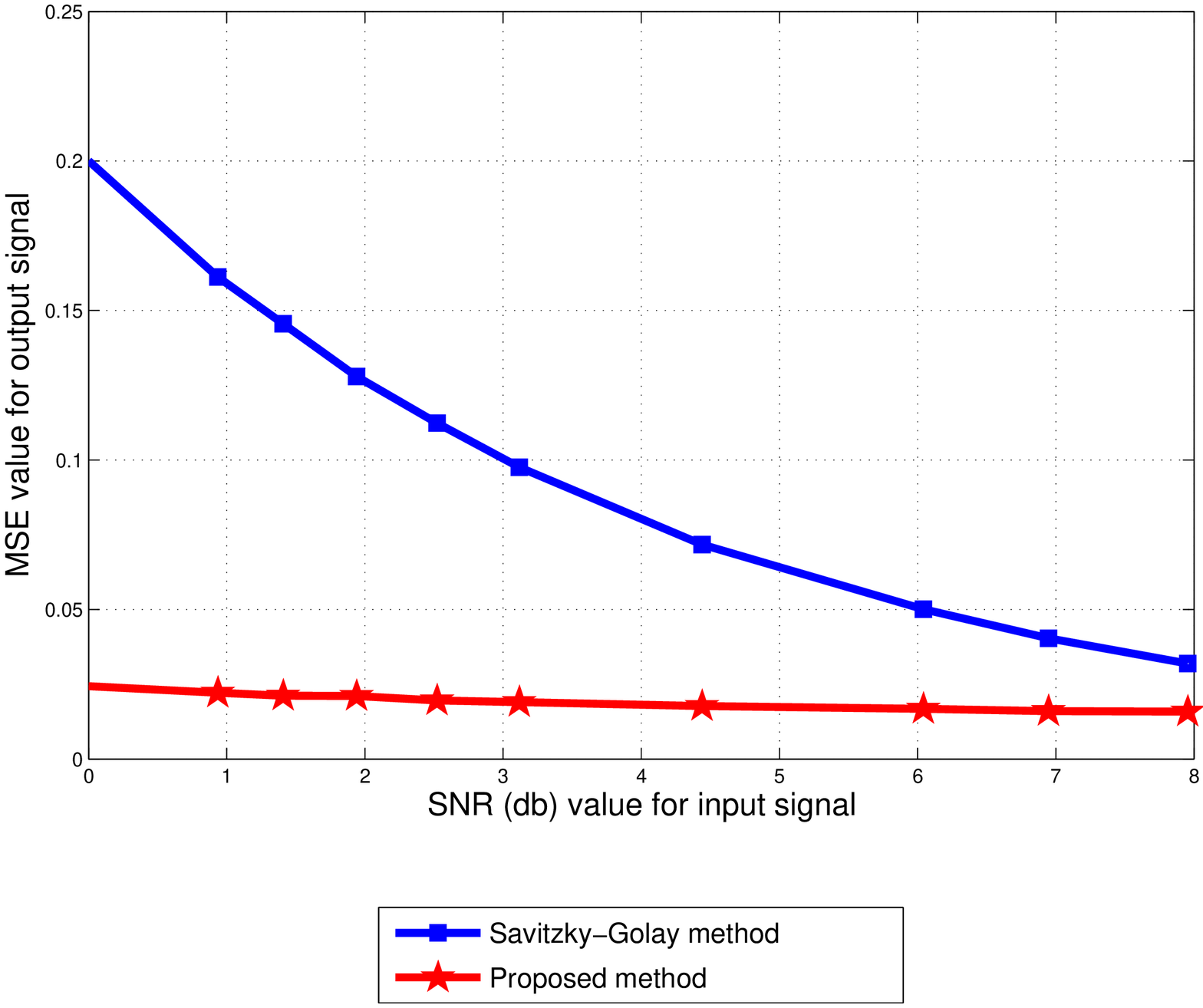} \label{sinusoidaleb} }
        \caption{ Choice of parameters: $4 \pi^2 a=0.01, b=0.03$ and $ \lambda=1.5$. The initial signal is $u_0(x) = \cos(5\pi x) + \cos(20 \pi x)$. \label{snrmsesinus}
}
\end{figure} 

\begin{figure}[h!]
	\centering
        \subfigure[ SNR values in Savitzky-Golay method in comparison to proposed method ]
	{\includegraphics[scale=0.3]{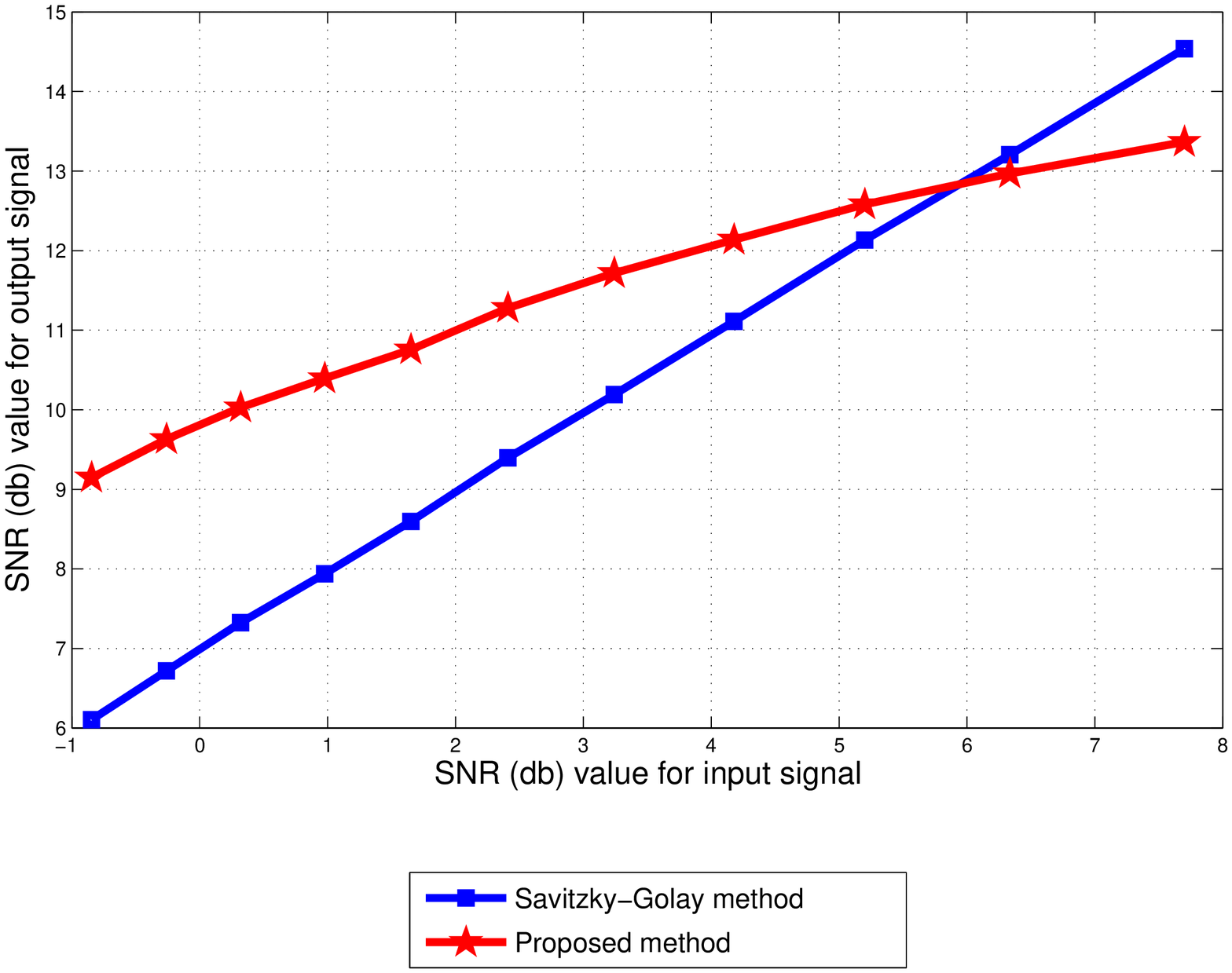}  \label{ecga}}
		\subfigure[ MSE values in Savitzky-Golay method in comparison to proposed method ]
	{\includegraphics[scale=0.3]{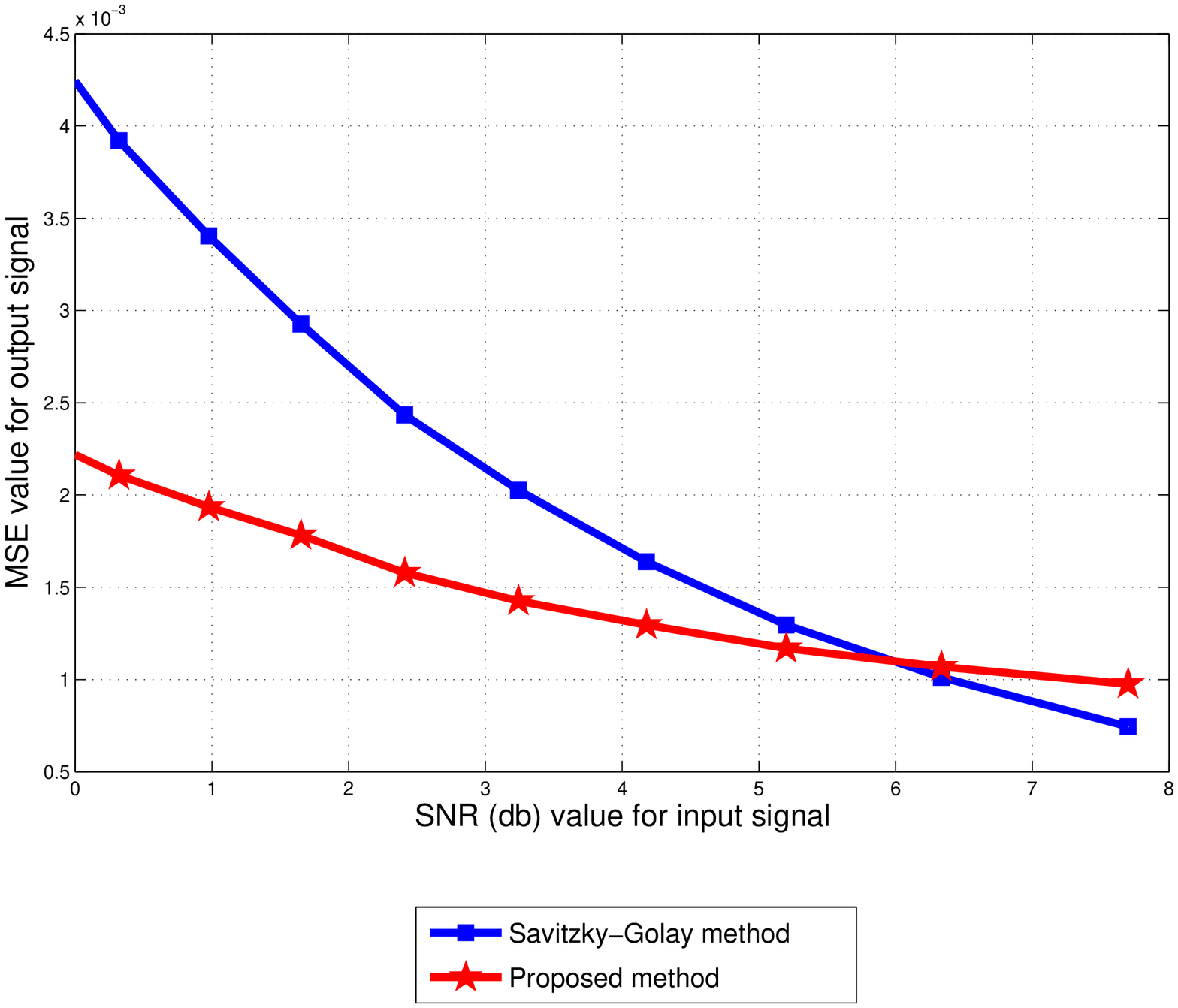} \label{ecgb}}
        \caption{ Choice of parameters: $4 \pi^2 a=0.005, b= 0.015$ and $\lambda=1.7$ . The initial signal $u_0(x)$ is an electrocardiogram (ECG) signal. \label{snrmseecg}}
\end{figure} 

\begin{figure}[h!]
	\centering
        \subfigure[ SNR values in Savitzky-Golay method in comparison to proposed method using finite difference scheme ]
	{\includegraphics[scale=0.3]{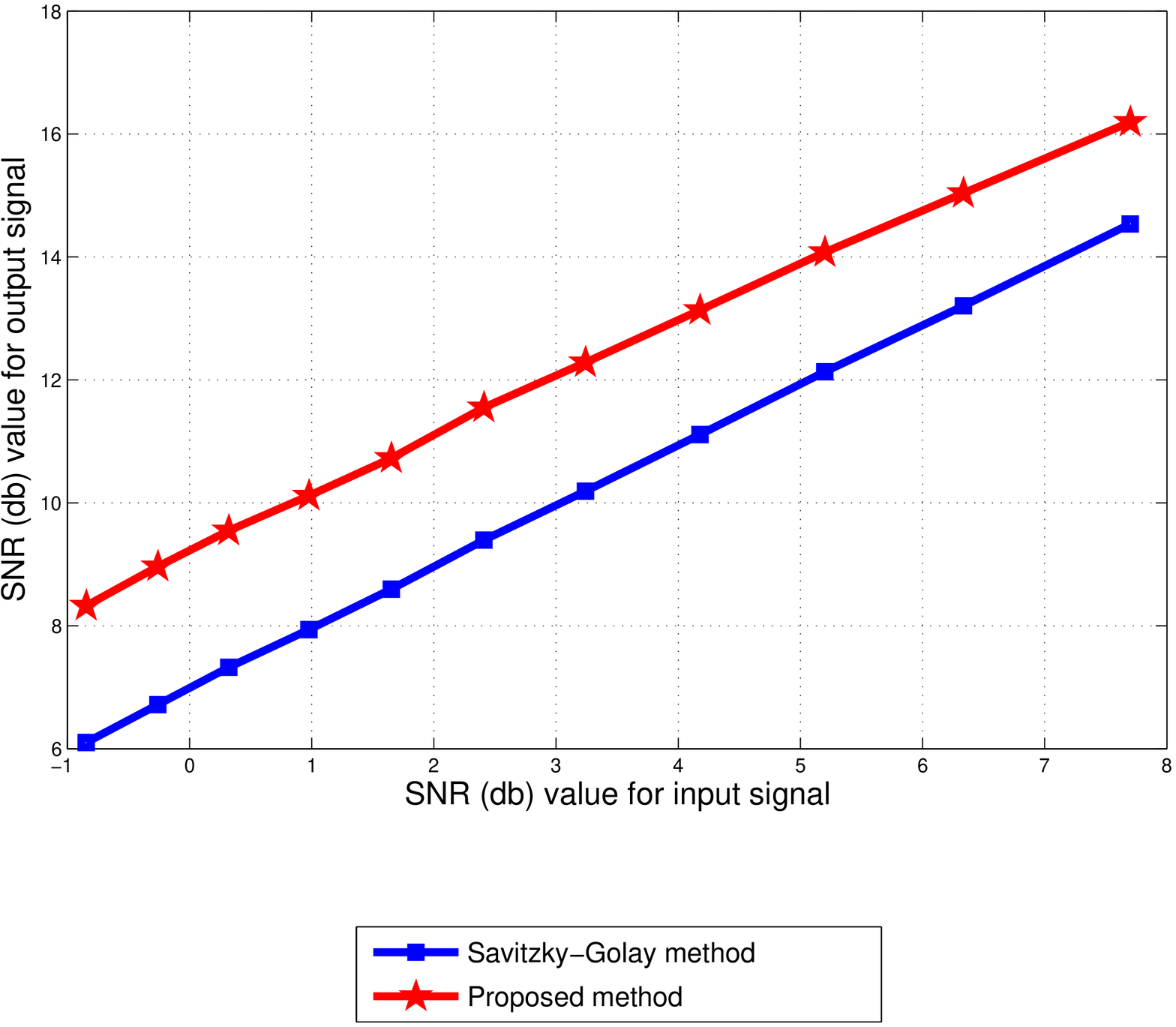}  \label{ecgdfa}}
	\subfigure[ MSE values in Savitzky-Golay method in comparison to proposed method ]
	{\includegraphics[scale=0.3]{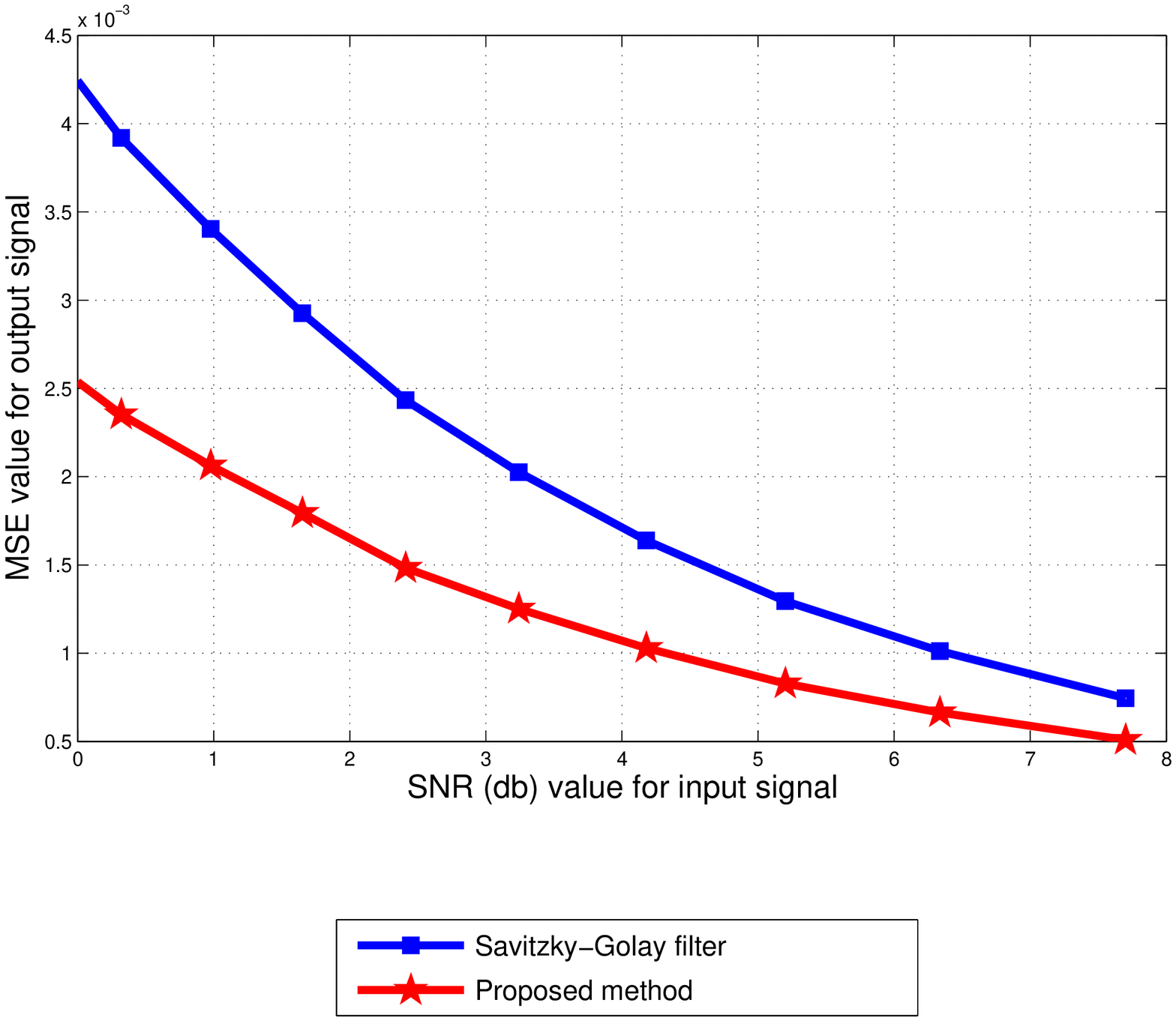} \label{ecgdfb}}
        \caption{ Choice of parameters: $a=0.5, b=0.3, \lambda=1/3$. The initial signal $u_0(x)$ is an electrocardiogram (ECG) signal. \label{snrmseecgfd}
}
\end{figure} 

\begin{figure}[h!]
	\centering
        \subfigure[ ]
	{\includegraphics[scale=0.4]{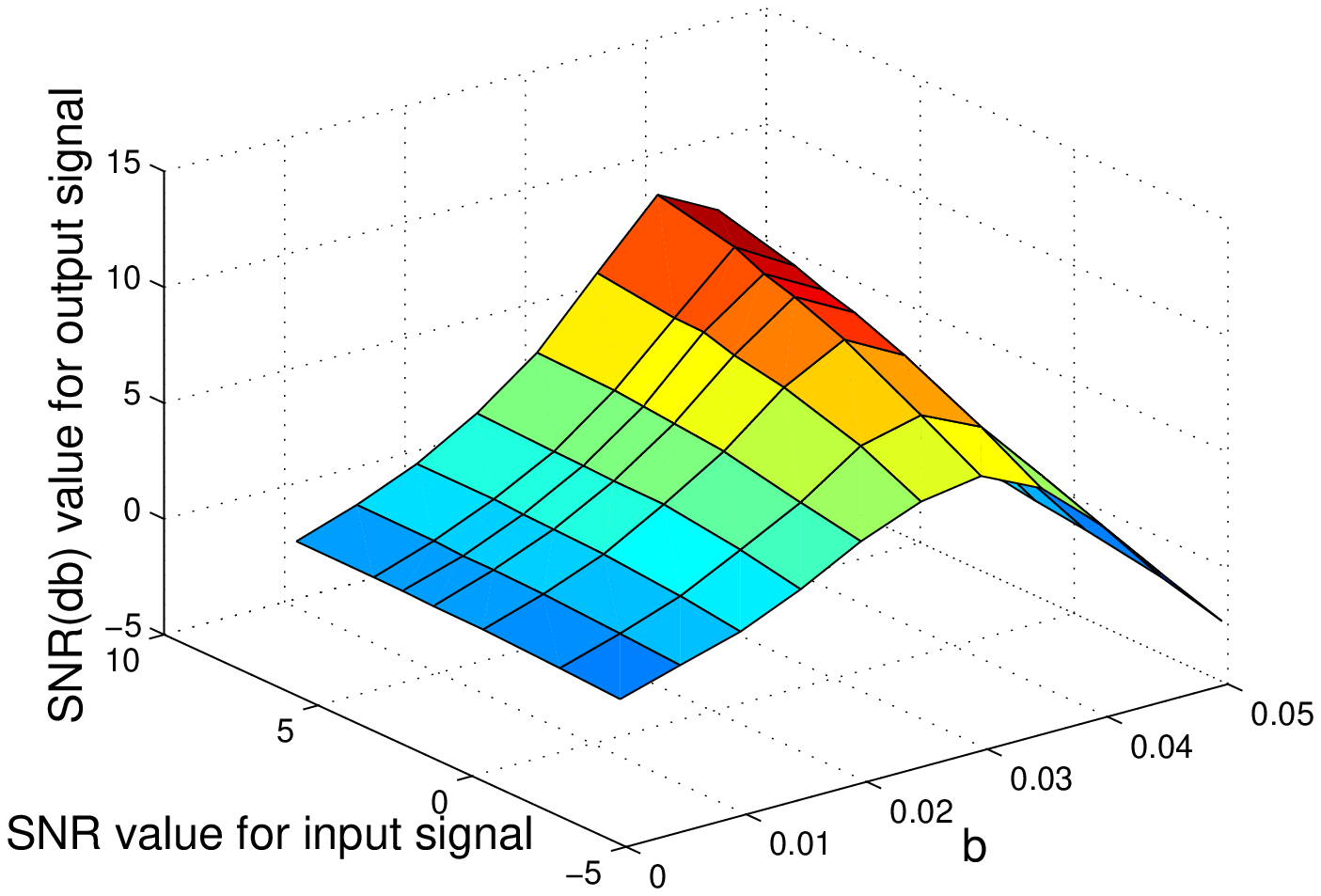}  \label{sinusoidaleaa} }
	\subfigure[  ]
	{\includegraphics[scale=0.4]{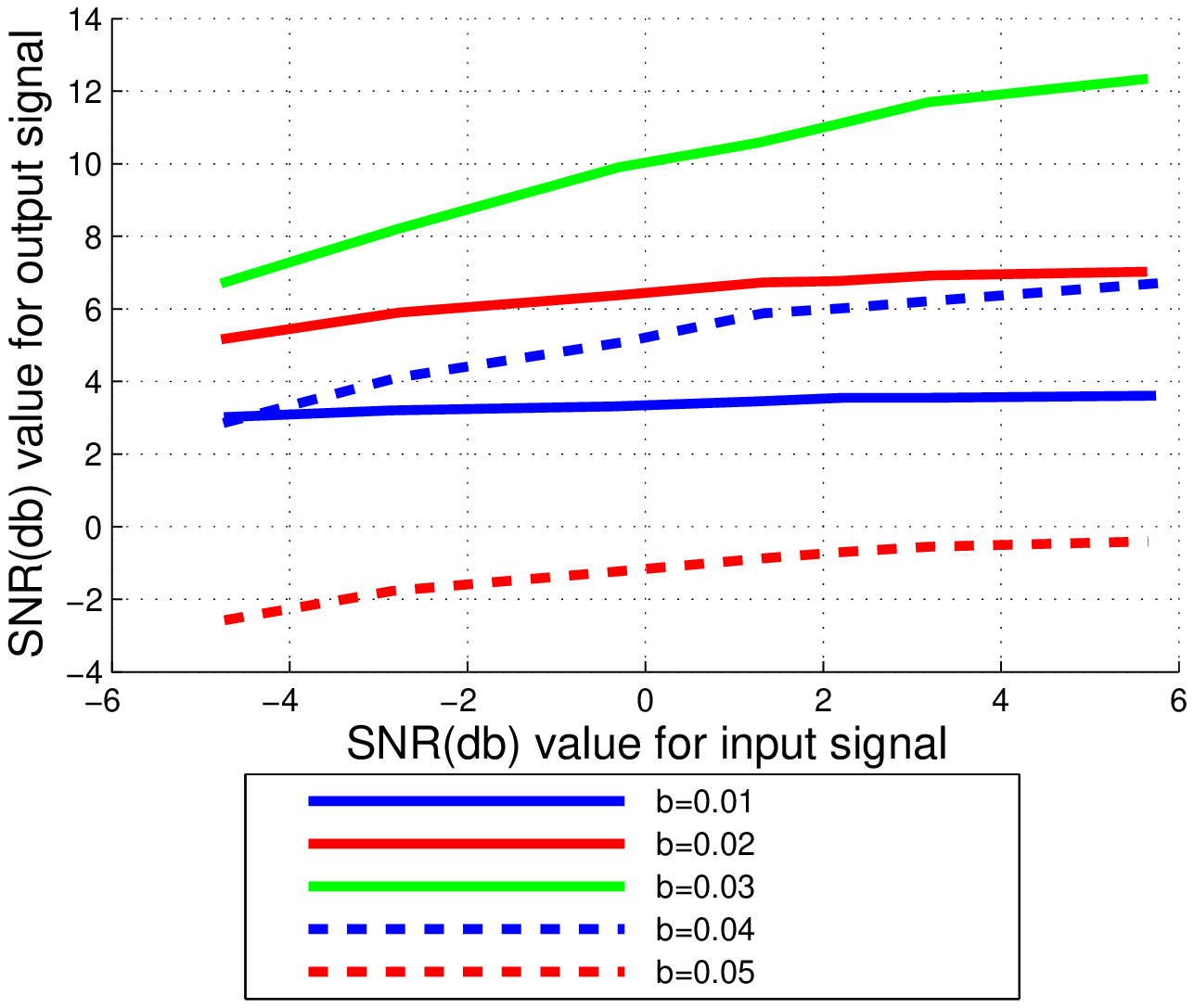} \label{sinusoidalebb} }
        \caption{ Choice of parameters: $ 4 \pi^2 a=0.01$ and  $\lambda=1.5$ . The initial signal is $u_0(x) = \cos(5\pi x) + \cos(20 \pi x)$. \label{snrmsesinusB}
}
\end{figure} 
\begin{remark}
 Let us briefly explain how SNR metrics allow us to optimize the choice of parameter $b$.
In figure \ref{snrmsesinusB}, we display the behaviour of SNR values for different values of $b$. As we can remark, the denoising will be most efficient  for $b \approx 0.03$. Hence visualization of the surface (figure \ref{sinusoidaleaa}) and curves (figure \ref{sinusoidalebb}) enables us to find  easily the parameter $b$ for optimal denoising. 
%In the same way, we obtain $b_{optim}$ for ECG signal and more generally for any signal. 
\end{remark}

\section{Concluding remarks \label{conclusion}} 

Our first aim was to introduce a fractal conservation law for denoising and contrast enhancement of signals. 
This device permits to reduce considerably the noise and to increase contrasts simultaneously. The study showed that our filter eliminates the high frequencies and amplifies the low/medium frequencies. In this paper, we also discussed the choice of parameters $a,b$ and $\lambda$. %\textcolor{red}{ We exhibit the term $2 \left(\frac{b}{a}\right)^{\frac{1}{2-\lambda}}$ responsible for denoising and we show that $b$ is the regulator of contrasts. }
%\textcolor{red}{ Since the coefficient $a$ control the Laplacian term and $b$ the fractal operator which has a antidiffusive effect, it is clear that the denoising depends on $a$ and that the regulator of contrasts is $b$. We showed that the role of the parameters $\lambda$ is not major except if this one is approaches $0$ or $2$.}

This equation has been implemented using both a finite difference scheme and the fast Fourier transform.  Various well-known measuring metrics have been used to compare our method with the Savitzky-Golay filter.   
Results showed the good performance of our model. Moreover, the analysis highlighted that, depending on the considered signal, it may be more suitable to use the finite difference scheme or the FFT algorithm when the SNR is high. Obviously, for a sinusoidal type signal, it is preferable to use the FFT, whereas for a signal like step functions it is better to implement the fractal equation with finite difference method. Nevertheless, no matter the algorithm used, the fractal consersation law \eqref{EDP} is an interesting and natural method for denoising and contrast enhancement. 

These satisfying properties for signal processing encourages us to implement it for image enhancement: this study is in progress.

\section{Acknowledgements} 
We thank Bijan Mohammadi for advice on numerical schemes and for helpful comments. P. Azerad and A. Bouharguane are supported by the ANR MATHOCEAN ANR-08-BLAN-0301-02.

\bibliographystyle{plain}

\end{document}